\documentclass{amsart}
\usepackage[utf8]{inputenc}

\usepackage{amsfonts}
\usepackage{amsmath}
\usepackage{amssymb}
\usepackage{amsthm}
\usepackage{cite} 
\usepackage{dutchcal}
\PassOptionsToPackage{hyphens}{url}\usepackage[colorlinks=true,citecolor={blue},linkcolor = purple]{hyperref}
\usepackage{enumerate}
\usepackage{enumitem}
\usepackage{extarrows}
\usepackage[margin=1in]{geometry}
\usepackage{mathdots}
\usepackage{mathrsfs}
\usepackage{tikz}
\usepackage{tikz-cd}
\usepackage{todonotes}
\usepackage{verbatim}
\usepackage{mathtools} 

\usepackage{wasysym} 

\usepackage{comment} 

\newcommand{\scr}[1]{\mathscr{#1}}

\DeclareFontFamily{U}{wncy}{}
    \DeclareFontShape{U}{wncy}{m}{n}{<->wncyr10}{}
    \DeclareSymbolFont{mcy}{U}{wncy}{m}{n}
    \DeclareMathSymbol{\Sha}{\mathord}{mcy}{"58}
    \DeclareMathSymbol{\Den}{\mathord}{mcy}{"44}
    \DeclareMathSymbol{\Num}{\mathord}{mcy}{"4E}

\theoremstyle{definition}
\newtheorem{theorem}{Theorem}[section]
\newtheorem{lemma}[theorem]{Lemma}

\newtheorem{corollary}[theorem]{Corollary}
\newtheorem{proposition}[theorem]{Proposition}

\newtheorem{definition}[theorem]{Definition}
\newtheorem{example}[theorem]{Example}
\newtheorem{remark}[theorem]{Remark}
\newtheorem{question}[theorem]{Question}

\newtheorem{situation}[theorem]{Situation}

\newtheorem{warning}[theorem]{Warning}

\newcommand{\End}{\text{End}}

\newcommand{\Afftrans}{\text{Aff}}

\newcommand{\UConf}{\text{UConf}}
\newcommand{\Conf}{\text{Conf}}

\newcommand{\OO}{\mathcal{O}}
\newcommand{\Aut}{\mathrm{Aut}}
\newcommand{\AAut}{\mathrm{Aut}}

\newcommand{\Spec}{\mathrm{Spec}\,}
\newcommand{\Proj}{\mathrm{Proj}\,}
\newcommand{\Tor}{\mathrm{Tor}}
\newcommand{\ZZ}{\mathbb{Z}}

\newcommand{\Aff}{\mathbb{A}}

\newcommand{\TMon}{\text{Twist}\Gen}

\newcommand{\Sch}[1]{(Sch/#1)}

\newcommand{\cal}[1]{\mathcal{#1}}

\renewcommand{\frak}[1]{\mathfrak{#1}}

\newcommand{\SSpec}[2]{\underline{\mathrm{Spec}}_{#1}\, #2}
\newcommand{\PProj}[2]{\underline{\mathrm{Proj}}_{#1}\, #2}

\newcommand{\localindex}{{\mathcal{i}}}
\newcommand{\indexsheaf}{{\mathcal{I}}}
\newcommand{\classicalindex}{{\mathrm{Index}}}

\definecolor{sebgreen1}{rgb}{0.019,0.317,0.149}
\definecolor{sebgreen2}{rgb}{0.784,0.952,0.780}

\newcommand{\Leo}[2][inline]{\todo[linecolor=purple,backgroundcolor=purple!25,bordercolor=purple,#1,shadow,author=Leo]{#2}} 

\renewcommand{\epsilon}{\varepsilon}

\newcommand{\e}{\epsilon}

\renewcommand{\hat}[1]{\widehat{#1}}

\newcommand{\gb}{{\mathfrak{b}}}
\newcommand{\gc}{{\mathfrak{c}}}

\newcommand{\gm}{{\mathfrak{m}}}

\newcommand{\gp}{{\mathfrak{p}}}
\newcommand{\gq}{{\mathfrak{q}}}

\newcommand{\gl}{{\mathfrak{l}}}

\def\Ocal{{\mathcal O}}

\renewcommand{\AA}{\mathbb{A}}

\newcommand{\CC}{\mathbb{C}}
\newcommand{\FF}{\mathbb{F}}
\newcommand{\GG}{\mathbb{G}}

\newcommand{\PP}{\mathbb{P}}
\newcommand{\QQ}{\mathbb{Q}}
\newcommand{\RR}{\mathbb{R}}

\counterwithin{figure}{section}

\newcommand{\stquot}[1]{{\left[ #1 \right]}}

\renewcommand{\tilde}[1]{\widetilde{#1}}

\usepackage{ifthen}
\newboolean{seriousversion}
\setboolean{seriousversion}{true}

\ifthenelse{\boolean{seriousversion}}{\excludecomment{joke}}{}

\renewcommand{\cong}{\simeq}

\newcommand{\dnd}{\nmid}
\newcommand{\id}{\mathrm{id}}

\newcommand{\ol}[1]{\overline{#1}}

\newcommand{\pr}{\mathrm{pr}}

\newcommand{\action}{\:\rotatebox[origin=c]{-90}{$\circlearrowright$}\:}

\newcommand{\Pic}{\text{Pic}}
\newcommand{\Hom}{\text{Hom}}
\newcommand{\HHom}{\underline{\text{Hom}}}
\newcommand{\Gen}{\mathcal{M}}
\newcommand{\NGen}{\mathcal{N}}
\newcommand{\PGen}{\mathbb{P}\Gen}

\newcommand{\VV}{\mathbb{V}}
\newcommand{\Sym}[1]{\text{Sym}^{#1}}

\newcommand{\WR}{\cal R}

\usepackage{stmaryrd}

\newcommand{\GL}{\text{GL}}
\newcommand{\PGL}{\text{PGL}}

\newcommand{\pb}{\ar[phantom, dr, very near start, "\ulcorner"]}

\newcommand{\jetsp}[1]{J_{#1}}

\newcommand{\DAfftrans}{\ensuremath{\Aff^k \rtimes \GG_m}}

\newcommand{\Ints}{\mathbb{Z}} 

\newcommand{\num}[1]{\langle #1 \rangle}

\newcommand{\IIsom}{\underline{\text{Isom}}}
\newcommand{\Mon}{\Gen}

\newcommand{\disc}{\operatorname{Disc}}

\makeatletter
\@namedef{subjclassname@2020}{%
  \textup{2020} Mathematics Subject Classification}
\makeatother

\title[The Scheme of Monogenic Generators II: Local Monogenicity and Twists]{The Scheme of Monogenic Generators II:\\Local Monogenicity and Twists}
\author[Sarah Arpin, Sebastian Bozlee, Leo Herr, Hanson Smith]{Sarah Arpin$^1$, Sebastian Bozlee$^2$, Leo Herr$^1$, Hanson Smith$^3$}
\date{\today}

\email{s.a.arpin@math.leidenuniv.nl}
\email{sebastian.bozlee@tufts.edu}
\email{l.s.herr@math.leidenuniv.nl}
\email{hsmith@csusm.edu}

\address{$^1$Universiteit Leiden, $^2$Tufts University, $^3$California State University San Marcos}

\begin{document}

\maketitle

\sloppy

\begin{abstract}
     This is the second paper in a series of two studying monogenicity of number rings from a moduli-theoretic perspective. By the results of the first paper in this series, a choice of a generator $\theta$ for an $A$-algebra $B$ is a point of the scheme $\mathcal{M}_{B/A}$. In this paper, we study and relate several notions of local monogenicity that emerge from this perspective. We first consider the conditions under which the extension $B/A$ admits monogenerators locally in the Zariski and finer topologies, recovering a theorem of Pleasants as a special case. We next consider the case in which $B/A$ is \'etale, where the local structure of \'etale maps allows us to construct a universal monogenicity space and relate it to an unordered configuration space. Finally, we consider when $B/A$ admits local monogenerators that differ only by the action of some group (usually $\mathbb{G}_m$ or $\mathrm{Aff}^1$), giving rise to a notion of twisted monogenerators. In particular, we show a number ring $A$ has class number one if and only if each twisted monogenerator is in fact a global monogenerator $\theta$.
\end{abstract}

\tableofcontents

\section{Introduction}\label{sec:intro}

We begin by recalling the essential points of the previous paper of this series\cite{abhs_paper_1}. Given an extension of commutative rings with identity (henceforth, rings) $B/A$, we say that that $B$ is \emph{monogenic over $A$} if there is an element $\theta \in B$ so that $B = A[\theta]$. Such an element is called a \emph{monogenerator}. Similarly, $B$ is said to be \emph{$k$-genic} over $A$
if there exists a tuple $(\theta_1, \ldots, \theta_k) \in B^k$ so that $B = A[\theta_1, \ldots, \theta_k]$. Such a tuple is a \emph{generating $k$-tuple}. We are motivated by the case of an extension of number rings $\ZZ_L/\ZZ_K$.

Such extensions of number rings are finite locally free over a Noetherian base. In fact all we need for our results are maps of schemes that are Zariski locally of this form. We gather these hypotheses into a common ``Situation" for convenience.

\begin{situation}\label{sit:gensetup_paper2}
Let $\pi : S' \to S$ be a finite locally free morphism of schemes of constant degree $n \geq 1$ with $S$ locally noetherian, and let $X \to S$ be a quasiprojective morphism (almost always $\Aff^1_S$ or $\Aff^k_S$). 
\end{situation}

In the preceding paper we prove the following representability result, which implies in particular that if $B$ is finite locally free over a Noetherian ring $A$, then there is a scheme that represents the monogenerators for $B$ over $A$.

\begin{theorem}[{{\cite[Proposition 2.3, Corollary 3.8]{abhs_paper_1}}}]
Let $\pi : S' \to S$ be as in Situation \ref{sit:gensetup_paper2}. Then

\begin{enumerate}
  \item There exists a smooth, quasiaffine $S$-scheme $\Gen_{X,S'/S}$ representing the contravariant functor on $S$-schemes
  \[
      (T \to S) \mapsto \left\{ \begin{tikzcd}
      S' \times_S T \ar[rr, "s"] \ar[rd] & & X \times_S T\ar[dl] \\
      & T 
      \end{tikzcd} \, \, \middle| \, \, s \text{ is a closed immersion.}\right\}.
  \]
  \item If $X = \Aff^1_S$, then $\Gen_{X,S'/S}$ is an affine $S$-scheme.
\end{enumerate}
\end{theorem}

We write $\Gen_{k, S'/S}$ for the case in which $X = \Aff^k_S$. We call the scheme $\Gen_{k,S'/S}$ the \emph{scheme of $k$-generators}. If $k = 1$, we write $\Gen_{S'/S}$ instead of $\Gen_{1,S'/S}$ and call it the \emph{scheme of monogenerators} or \emph{monogenicity space}. If $S' = \Spec B$ and $S = \Spec A$ are affine, we may write $\Gen_{k, B/A}$ or $\Gen_{B/A}$ instead.

In the case that $S = \Spec A$, $S' = \Spec B$, and $T = \Spec C$, standard universal properties imply that the $T$-points of $\Gen_{k,S'/S}$ are in natural bijection with the generating $k$-tuples of $B \otimes_A C$ over $C$. If we assume further that $T \simeq S$, we find that
the $S$-points of $\Gen_{k,B/A}$ are in bijection with generating $k$-tuples for $B$ over $A$.

By analogy with the affine case, we therefore say that the $S$-points of $\Gen_{k,S'/S}$ are
\emph{generating $k$-tuples} and the $S$-points of $\Gen_{S'/S}$ are \emph{monogenerators} for
$S' \to S$. Such a morphism is \emph{monogenic} if a monogenerator exists.

\subsection{Equations in local coordinates}

The scheme $\Gen_{1,S'/S}$ has a simple description in local coordinates on $S$ which we recall so that we may use it in computations. We start by noticing that $\Gen_{1,S'/S}$ is naturally a subscheme of another moduli scheme, the Weil Restriction.

\begin{definition}
Let $\pi : S' \to S$ be as in Situation \ref{sit:gensetup_paper2}. The \emph{Weil Restriction} of $X \times_S S'$ to $S$, denoted $\WR_{X,S'/S}$, is the scheme (unique up to isomorphism) which represents the contravariant functor
\[
  (T \to S) \mapsto \left\{ \begin{tikzcd}
      S' \times_S T \ar[rr, "s"] \ar[rd] & & X \times_S T\ar[dl] \\
      & T 
      \end{tikzcd} \right\}
\]
on $S$-schemes.
\end{definition}

We abbreviate $\WR_{X,S'/S}$ in a parallel fashion to $\Gen_{X,S'/S}$. It is proven in \cite[Theorem 1.3, Proposition 2.10]{weilrestnpatmcfaddin} that the Weil Restriction exists and is a quasiprojective $S$-scheme.
We prove in \cite[Proposition 2.3]{abhs_paper_1} that the natural map $\Gen_{X,S'/S} \to \WR_{X,S'/S}$ is a quasi-compact open immersion.

In the case that $B$ is a finite free Noetherian $A$-algebra with $A$-basis $e_1, \ldots, e_n$, things are simpler. It is easy to check that
\[
  \mathbb{A}^n_S \cong \WR_{1,S'/S}
\]
via the isomorphism sending $(x_1,\ldots,x_n)$ to the unique map $S' \times_S T \to \mathbb{A}^1_T = \Spec \mathscr{O}_{T} [t]$ over $T$ sending $t$ to $x_1e_1 + \cdots + x_ne_n$.

\begin{definition} \label{def:local_coords}
Suppose $B$ is a finite free Noetherian $A$-algebra with $A$-basis $e_1, \ldots, e_n$. Let $a_{ij}$ for $1 \leq i, j \leq n$ be the unique elements of $A[x_1, \ldots, x_n]$ so that we have
\[
  (x_1e_1 + \cdots + x_ne_n)^{i-1} = a_{i,1}e_1 + \cdots + a_{i,n}e_n
\]
in the ring $B[x_1,\ldots, x_n]$.
We call the matrix $M(e_1,\ldots,e_n) = [a_{ij}]_{1 \leq i,j \leq n}$ the \emph{matrix of coefficients} with respect to the basis $e_1,\ldots, e_n$. Its determinant $\localindex(e_1,\ldots,e_n) \in A[x_1, \ldots, x_n]$ is the \emph{local index form} with respect to the basis $e_1,\ldots, e_n$.
\end{definition}

\begin{theorem} ({\cite[Theorem 3.1]{abhs_paper_1}}) \label{thm:gen_local_coords}
With notation as above, $\Gen_{1,B/A}$ is the distinguished open subscheme of $\WR_{1,S'/S}$ cut out by the non-vanishing of the local index form.
In particular,
\[
  \Gen_{1,B/A} \cong \Spec A[x_1,\ldots,x_n, \localindex(e_1,\ldots,e_n)^{-1}].
\]
\end{theorem}

Additionally, we recall from \cite{abhs_paper_1} that the local index forms give the complement of $\Gen_{S'/S}$ in $\WR_{S'/S}$ a closed subscheme structure:

\begin{definition}[Non-monogenerators $\NGen_{S'/S}$]\label{defn:nonmonogens}
Let $\indexsheaf_{S'/S}$ be the ideal sheaf on $\WR_{S'/S}$ generated locally by local index forms. We call this the \emph{index form ideal}. Let $\NGen_{S'/S}$ be the closed subscheme of $\WR$ cut out by the vanishing of $\indexsheaf_{S'/S}$. We call this the \emph{scheme of non-monogenerators}, since its support is the complement of $\Gen_{S'/S}$ inside of $\WR_{S'/S}$.
\end{definition}

\bigskip

Since $\Gen_{k,S'/S}$ is a scheme, it is a sheaf in the fpqc topology on $\Sch{S}$. This invites a local study of monogenicity\footnote{`Monogeneity' is also common in the literature.} of $S' \to S$,
the subject of this paper.

\subsection{Results}

We identify and relate several ``local" notions of monogenicity. To guide the reader, their relationships are indicated in Figure \ref{fig:monogenicity_diagram}.

\begin{figure}
\begin{center}
\begin{tikzcd}
&\text{(Globally) Monogenic} \ar[d, Rightarrow]      \\
&\GG_m\text{-Twisted Monogenic} \ar[d, Rightarrow]         \\
&\Afftrans^1\text{-Twisted Monogenic} \ar[d, Rightarrow, "\text{\S\ref{ssec:afftwisted}}"]        \\
\text{\centering\parbox{.85in}{Monogenic at completions}\,\,} & \text{Zariski-Locally Monogenic}\arrow[d,Rightarrow] \ar[r,Leftrightarrow, "\text{Thm.~\ref{thm:locally_mono_is_mono_over_points}}"] \ar[l, Leftrightarrow, "\text{Thm.~\ref{thm:locally_mono_is_mono_over_points}}", swap] & \text{\parbox{.85in}{Monogenic at points}} \\
\text{\parbox{.85in}{Fpqc-Locally Monogenic}} \ar[r,Leftrightarrow,"\text{Thm.~\ref{thm:monogenic_over_geometric_points}}"] & \text{\parbox{.85in}{\'Etale-Locally Monogenic}} \ar[r,Leftrightarrow,"\text{Thm.~\ref{thm:monogenic_over_geometric_points}}"]
&\text{\parbox{1in}{Monogenic at geometric points}} \ar[ul, Rightarrow, dashed, "\text{Cor.~\ref{cor:inf_residue_fields_mono_at_points}}"']&
\end{tikzcd}
\end{center}
\caption{A guide to notions of monogenicity and their relationships. The vertical dashed implication holds under additional hypotheses, see Cor.~\ref{cor:inf_residue_fields_mono_at_points}.}
\label{fig:monogenicity_diagram}
\end{figure}
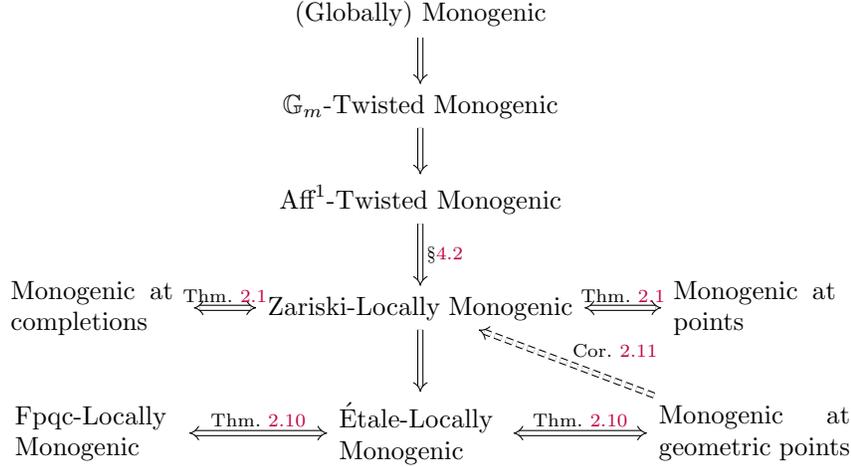

A sheaf theoretic notion of local monogenicity immediately presents itself.\footnote{We remark that Zariski local monogenicity is equivalent to the condition of ``local monogenicity" considered by Bhargava et. al. in \cite{AlpogeBhargavaShnidman} but strictly weaker than their condition of ``no local obstruction to monogenicity."}

\begin{definition}\label{def:sheafmono_paper2}
Let $\tau$ be a subcanonical Grothendieck topology on schemes, for example the Zariski, Nisnevich, \'etale, fppf, or fpqc topologies.
We say that $S'/S$ is $\tau$-\emph{locally $k$-genic} if the sheaf $\Gen_{k, S'/S}$ is locally non-empty in the topology $\tau$.
I.e., there is a $\tau$-cover $\{ U_i \to S \}_{i \in I}$ of $S$ such that $\Gen_{k, S'/S}(U_i)$ is non-empty for all $i \in I$.
\end{definition}

The notions of $\tau$-local monogenicity are considered in Section \ref{sec:locgen}, and we find that these reduce to just two notions of ``local monogenicity."

\begin{theorem}
Let $\pi : S' \to S$ be as in Situation \ref{sit:gensetup_paper2}.
\begin{enumerate}
\item (Theorem \ref{thm:locally_mono_is_mono_over_points}) The following are equivalent:
  \begin{enumerate}
    \item $\pi$ is locally monogenic in the Zariski topology;
    \item $\pi$ is ``monogenic at completions," i.e. for all points $x$ of $S$, we have that $S' \times_S \Spec \hat{\mathscr{O}}_{S,x}$ is monogenic;
    \item $\pi$ is ``monogenic at points," i.e. for each point $x$ of $S$ with residue field $k(x)$, we have that $S' \times_S \Spec k(x) \to \Spec k(x)$ is monogenic.
  \end{enumerate}
\item (Theorem \ref{thm:monogenic_over_geometric_points}) The following are equivalent:
  \begin{enumerate}
    \item $\pi$ is locally monogenic in the \'etale topology;
    \item $\pi$ is locally monogenic in the fpqc topology;
    \item $\pi$ is ``monogenic at geometric points," i.e. for all algebraically closed fields $k$ and maps $\Spec k \to S$, we have that $S' \times_S \Spec k \to \Spec k$ is monogenic.
  \end{enumerate}
\end{enumerate}
\end{theorem}

We then use the structure of finite algebras over fields to classify monogenicity at points. In particular, we recover Pleasants's characterization \cite[Theorems 1 and 2]{Pleasants} of monogenicity at completions as a corollary. 

\begin{theorem}\label{thm:local_artin_monogenicity}
Let $S' \to S$ be induced by $k \to B$ where $k$ is a field and $B$ is a local Artinian $k$-algebra with residue field $\ell$ and maximal ideal
$\gm$. Then $S' \to S$ is monogenic if and only if
\begin{enumerate}
  \item $\Spec \ell \to \Spec k$ is monogenic;
  \item $\dim_{\ell} \gm/\gm^2 \leq 1$; and
  \item If $\dim_{\ell} \gm/\gm^2 = 1$ and $\ell/k$ is inseparable, then
  \[
    0 \to \gm/\gm^2 \to B/\gm^2 \to \ell \to 0
  \]
  is a non-split extension.
\end{enumerate}
\end{theorem}

\begin{theorem}\label{thm:artin_monogenicity}
Suppose $S' \to S$ is induced by $k \to A$ where $k$ is a field and $B$ is an Artinian $k$-algebra. Write $B = \prod_i B_i$ where the $B_i$ are local artinian $k$-algebras with respective residue fields $\ell_i$. Then $S' \to S$ is monogenic if and only if
\begin{enumerate}
  \item \label{artinmono1} $\Spec B_i \to S$ is monogenic for each $i$;
  \item \label{artinmono2} for each finite extension $\ell$ of $k$, $S'$ has fewer points with residue field isomorphic to $\ell$ than $\mathbb{A}^1_S$.
\end{enumerate}
\end{theorem}

\bigskip

In Section \ref{sec:confspsmonetalecase}, we consider monogenicity spaces of \'etale $S' \to S$, a salient case since extensions of number rings are generically \'etale and
such \'etale maps share a common local structure:
Finite \'etale maps are \'etale-locally isomorphic to the trivial $n$-sheeted cover
$S \sqcup \cdots \sqcup S \to S$. The latter has $\Gen_{S'/S}$ equal to the configuration
space of $n$ distinct points in $\AA^1$. Therefore we may interpret monogenicity spaces as twisted generalizations of configuration spaces, at least when $S' \to S$ is \'etale.

Using the fact\cite[04HN]{sta} that every finite \'etale map of degree $n$ is pulled back from the morphism of stacks $B\Sigma_{n - 1} \to B\Sigma_n$, we then construct a monogenicity space $\Gen_{B\Sigma_{n-1}/B\Sigma_n}$.

\begin{theorem}\label{thm:three_isoms}
There are isomorphisms
\begin{align*}
  \WR_{B\Sigma_{n-1}/B\Sigma_n} &\cong [\AA^n / \Sigma_n], \\
  \NGen_{B\Sigma_{n-1}/B\Sigma_n}  &\cong [\hat{\Delta} / \Sigma_n],  \\
  \text{ and } \Gen_{B\Sigma_{n-1}/B\Sigma_n} &\cong [(\AA^n - \hat{\Delta}) / \Sigma_n]
\end{align*}
in which the action by $\Sigma_n$ is in each case the appropriate restriction of the permutation action on coordinates of $\AA^n$, and $\hat{\Delta}$ denotes the ``fat diagonal" of $\AA^n$, the locus where some pair of coordinates coincide.
\end{theorem}

In particular, the $\CC$-points of $\Gen_{B\Sigma_{n-1}/B\Sigma_n}$ coincide with the points of the unordered configuration space of $n$ points in $\CC$. This monogenicity space is universal for \'etale maps in the sense that the monogenicity space $\Gen_{S'/S}$ of each \'etale $S' \to S$ is pulled back from $\Gen_{B\Sigma_{n-1}/B\Sigma_n}$.

We then consider several examples enabled by the structure of $\Gen_{S'/S}$ in the \'etale case, among them a connection to braid groups, a construction of the moduli space of genus zero pointed curves from monogenicity spaces, the monogenicity space of a $G$-torsor, and the monogenicity space of an isogeny of elliptic curves.

We remark that the monogenicity space $\Gen_{B\Sigma_{n-1}/B\Sigma_n}$ appears to be a scheme theoretic enhancement of the universal spaces of \cite[Section 3]{hansen_coverings} and \cite[Section 1]{gorin_lin} in the category of topological spaces. Meanwhile, the universal monogenicity space $\Gen_{*/BG}$ for $G$-torsors appears to be a scheme theoretic enhancement of the space $B_1G$ considered in \cite[Theorem 1.3]{duchamp_hain}. 

\bigskip

A local-to-global sequence is missing for $\tau$-local monogenicity. Yet there are natural group actions on $\Gen_{S'/S}$, in particular by $\GG_m$ and $\Afftrans^1$. In Section \ref{sec:twistedmon} we study ``twisted" versions of monogenicity in which $S' \to S$
has local monogenerators that differ from each other by the action of such groups.
More precisely:
\begin{definition}[Twisted Monogenerators]

A ($\GG_m$)-twisted monogenerator for $B/A$ is:
\begin{enumerate}
    \item a Zariski open cover $\Spec A = \bigcup_i D(f_i)$ for elements $f_i \in A$,
    \item a system of ``local'' monogenerators $\theta_i \in B\left[f_i^{-1}\right]$ for $B[f_i^{-1}]$ over $A[f_i^{-1}]$, and
    \item units $a_{ij} \in A\left[f_i^{-1}, f_j^{-1}\right]^\ast$
\end{enumerate}
such that
\begin{itemize}
    \item for all $i,j$, we have $a_{ij} . \theta_j = \theta_i,$
    \vspace{.1 cm}
    \item for all $i,j,k$, the ``cocycle condition'' holds:
\[a_{ij} . a_{jk} = a_{ik}.\]
\end{itemize}

Two such systems $\{(a_{ij}), (\theta_i)\}$, $\{(a_{ij}'), (\theta_i')\}$ are \emph{equivalent} if, after passing to a common refinement of their respective covers, there is a global unit $u \in A^\ast$ such that $u\cdot a_{ij} = a'_{ij}$ and $u \cdot \theta_i = \theta_i'$. 
Likewise $B/A$ is $\Afftrans^1$-\emph{twisted monogenic} if there is a cover with $\theta_i$'s as above, but with the units in item \eqref{eitem:unitstwistedmon} replaced by pairs $a_{ij}, b_{ij} \in A \left[f_i^{-1}, f_j^{-1}\right]$ such that each $a_{ij}$ is a unit and $a_{ij} \theta_j + b_{ij} = \theta_i$.
\end{definition}

Under certain hypotheses, we show:

\begin{itemize}[leftmargin= 1.2in]
    \item[Proposition \ref{prop:affequivlinebund}:] $B/A$ is $\GG_m$-twisted monogenic if and only if it is $\Afftrans^1$-twisted monogenic.
    
    \item[Theorem \ref{thm:twistedClassNumb1}:] The class number of a number ring $\Ints_K$ is one if and only if all twisted monogenic extensions of number rings $\Ints_L/\Ints_K$ are in fact monogenic.
    
    \item[Remark \ref{rmk:affineequivvstwistedmon}:] There is a local-to-global sequence relating affine equivalence classes of monogenerators with global monogenerators as above.
    
    \item[Theorem \ref{thm:GmandAffrepablepreview}:] There are moduli spaces of $\GG_m$ and $\Afftrans^1$-twisted monogenerators analogous to $\Gen_{S'/S}$. 
    
    \item[Theorem \ref{thm:fintwistedmons}:] There are finitely many twisted monogenerators up to equivalence.
\end{itemize}

We remark that a $\GG_m$-twisted monogenerator is equivalent to an embedding $S'$ over $S$ into a line bundle $L$ on $S$. Such embeddings into line bundles were considered for topological spaces in \cite{lonsted}.
\bigskip

Section \ref{sec:examples} concludes with ample examples of the scheme of monogenerators
and the various interactions between the forms of local monogenicity.

To avoid repetition, we invite the reader to consult the first paper in this series for a more detailed survey of the relevant literature.

\subsection{Acknowledgements}

The second author would like to thank David Smyth and Ari Shnidman for their support and for helpful conversations.

The third author would like to thank Gebhard Martin, the mathoverflow community for \cite{370972} and \cite{377840}, Tommaso de Fernex, Robert Hines, and Sam Molcho. Tommaso de Fernex looked over a draft and made helpful suggestions about jet spaces. The third author thanks the NSF for providing partial support by the RTG grant \#1840190. 

The fourth author would like to thank Henri Johnston and Tommy Hofmann for help with computing a particularly devious relative integral basis in Magma.

All four authors would like to thank their graduate advisors Katherine E. Stange (first and fourth authors) and Jonathan Wise (second and third authors). This project grew out of the fourth author trying to explain his thesis to the second author in geometric terms. We are greatful to Richard Hain for making us aware of the literature on monogenicity for Stein spaces and topological spaces. For numerous computations throughout, we were very happy to be able to employ Magma \cite{Magma} and SageMath \cite{Sage}. For a number of examples the \cite{lmfdb} was invaluable. Finally, we are thankful for Lily.


\section{Local monogenicity}\label{sec:locgen}

\subsection{Zariski-local monogenicity}

This section shows Zariski-local monogenicity can be detected over points and completions as spelled out in Remark \ref{rmk:local_notions}. We will make frequent use of the vocabulary and notation of \cite[Section 3]{abhs_paper_1}.

\begin{theorem} \label{thm:locally_mono_is_mono_over_points}
The following are equivalent:
\begin{enumerate}
    \item\label{1local} $\pi : S' \to S$ is Zariski-locally monogenic.
    
    \item\label{2local} There exists a family of maps $\{ f_i : U_i \to S \}$
    such that
    \begin{enumerate}
      \item the $f_i$ are jointly surjective; 
      \item for each point $p \in S$, there is an index $i$ and point
            $q_p \in f_i^{-1}(p)$ so that $f_i$ induces an isomorphism $k(p) \to k(q_i)$; 
      \item $S' \times_S U_i \to U_i$ is monogenic for all $i$.
    \end{enumerate}
    \item\label{3local} $\pi : S' \to S$ is \emph{monogenic over points}, i.e., $S' \times_S \Spec k(p) \to \Spec k(p)$ is monogenic for each point $p \in S$.
\end{enumerate}
\end{theorem}
\begin{proof}

\eqref{1local}$\implies$\eqref{2local}: Choose the $U_i$ to be a Zariski cover on which $S' \to S$ is monogenic.

\eqref{2local}$\implies$\eqref{3local}: Suppose such a cover $\{ f_i : U_i \to S \}$ is given. For each $i$, let $\sigma_p : \Spec k(p) \to U_i$ be the section through $q_p$. monogenicity is preserved by pullback on the base, so pulling back $S' \times_S U_i \to U_i$ along $\sigma_p$ implies \eqref{3local}.

\eqref{3local}$\implies$\eqref{1local}:
Let $p \in S$ be a point with residue field $k(p)$ and let $\theta$ be a monogenerator for $S' \times_S \Spec k(p) \to \Spec k(p)$. We claim that $\theta$ extends to a monogenerator over an open subset $U \subseteq S$ containing $p$, from which \eqref{1local} follows.
The claim is Zariski local, so assume $S = \Spec A$ and $\pi_*\OO_{S'} \simeq \bigoplus_{i=1}^n \OO_S e_i$ is globally free. The Weil Restriction $\WR_{S'/S}:= \HHom_S(S', X')$ is then isomorphic to affine space $\Aff^n_S$.

We first extend $\theta$ to a section of $\WR_{S'/S}$. The monogenerator entails a point $\theta : \Spec k(p) \to \Gen_{S'/S} \subseteq \WR_{S'/S} \simeq \Aff^n_S$, i.e. $n$ elements $\overline x_1, \dots, \overline x_n \in k(p)$. Choose arbitrary lifts $x_i \in A_{(p)}$ of $\overline{x}_i$. The $n$ elements $x_i$ must have a common denominator, so we have  $x_1, \dots, x_n \in A\left[f^{-1}\right]$ for some $f$. Thus our point $\theta : \Spec k(p) \to \Aff^n_S$ extends to $\widetilde{\theta} : D(f) \to \WR_{S'/S}$ for some distinguished open neighborhood $D(f) \subseteq S$ containing $p$.

Finally, we restrict $\widetilde{\theta}$ to a section of $\Gen_{S'/S}$. The monogenicity space $\Gen_{S'/S}$ is an open subscheme of the Weil Restriction $\WR_{S'/S}$, so $\widetilde{\theta} : D(f) \to \WR_{S'/S}$ restricts to a monogenerator $\widetilde{\theta}|_U : U \to \Gen_{S'/S}$ where $U = \widetilde{\theta}^{-1}(\Gen_{S'/S}) \subset D(f)$. By hypothesis, $p \in U$, so $\widetilde{\theta}|_U$ is the desired extension of $\theta$.
\end{proof}

\begin{remark}
The same proof shows that $S' \to S$ is Zariski-locally $k$-genic if and only if its fibers $S' \times_S \Spec k(p) \to \Spec k(p)$ are $k$-genic.
\end{remark}

\begin{remark}
\label{rmk:local_notions}
Item \eqref{2local} of Theorem \ref{thm:locally_mono_is_mono_over_points} implies the following are also equivalent to Zariski-local monogenicity:
\begin{enumerate}
    \item\label{def:local_notions_localrings} $S' \to S$ is ``monogenic at local rings," i.e, for each point $p$ of $S$, $S' \times_S \Spec \OO_{S,p} \to \Spec \OO_{S,p}$ is monogenic.
    \item\label{def:local_notions_completions} $S' \to S$ is ``monogenic at completions," i.e., for each point $p$ of $S$, $S' \times_S \Spec \hat{\OO}_{S,p} \to \Spec \hat{\OO}_{S,p}$ is monogenic, where $\hat{\OO}_{S,p}$ denotes the completion of ${\OO}_{S,p}$ with respect to its maximal ideal.
    \item $S' \to S$ is locally monogenic in the Nisnevich topology as in Definition \ref{def:sheafmono_paper2}.
\end{enumerate}
\end{remark}

\begin{corollary}[{\cite[Proposition III.6.12]{LocalFields}}]
Let $S'\to S$ be an extension of local rings inducing a separable extension of residue fields. Then $S'$ is monogenic over $S$.
\end{corollary}

\begin{proof}

Use the equivalence of item \eqref{def:local_notions_localrings} in Remark \ref{rmk:local_notions} and \eqref{3local} in Theorem \ref{thm:locally_mono_is_mono_over_points}.
\end{proof}

We now recall some ideas in order to compare with related results in the number ring case.

\begin{definition}\label{def: commonindexdivisor}
Let $\ZZ_L / \ZZ_K$ be an extension of number rings. Given $\theta \in \ZZ_L$ generating $L/K$, we write $\classicalindex_{\ZZ_L/\ZZ_K}(\theta)$ for the index $[\ZZ_L : \ZZ_K[\theta]]$. A non-zero prime of $\gp\subset \ZZ_K$
is a \textit{common index divisor}\footnote{Common index divisors are also called \textit{essential discriminant divisors} and \textit{inessential or nonessential discriminant divisors}. The shortcomings of the English nomenclature likely come from what Neukirch \cite[page 207]{Neukirch} calls ``the untranslatable German catch phrase [...] `au{\ss}erwesentliche Diskriminantenteile.'" Our nomenclature is closer to Fricke's `ständiger Indexteiler.'} for the extension $\ZZ_L/\ZZ_K$ if $\classicalindex_{\ZZ_L/\ZZ_K}(\theta)\in\gp$ for every $\theta\in \ZZ_L$ generating $L/K$.
\end{definition}

Common index divisors are exactly the primes $\gp$ whose splitting in $\ZZ_L$ cannot be mirrored by irreducible polynomials in $k(\gp)[x]$; see \cite{Hensel1894} and \cite{Pleasants}.

We recall\cite[Remark 3.11]{abhs_paper_1} that if $U$ is an open affine subscheme of $S$ on which $S'$ is free with basis $e_1, \ldots, e_n$ then $\classicalindex_{\ZZ_L/\ZZ_K}(x_1e_1 + \cdots + x_ne_n)$ is a local index form on $U$. 

Restating the property of being monogenic at points in terms of the index form, we obtain a generalization of the notion of having no common index divisors:

\begin{proposition}
\label{prop:mono_over_points_index_form}
$S' \to S$ is monogenic over points if and only if for each point $p$ of $S$ and local index form $\localindex$ around $p$, there is a tuple $(x_1, \ldots, x_n) \in k(p)^n$ such that $\localindex(x_1, \ldots, x_n)$ is nonzero in $k(p)$.
\end{proposition}

\begin{proof}
$S' \times \Spec k(p) \to \Spec k(p)$ is monogenic precisely when
$\Spec k(p)[x_1,\ldots,x_n, \localindex(x_1,\ldots, x_n)^{-1}]$ has a $k(p)$
point. 
\end{proof}

Immediately we recover an explicit corollary validating the generalization: 

\begin{corollary}\label{cor:Zar_local_mon_index_div}
Suppose $S' \to S$ is induced by an extension of number rings $\ZZ_L/\ZZ_K$. Then $S' \to S$ is Zariski-locally monogenic if and only if there
are no common index divisors for $\ZZ_L/\ZZ_K$.
\end{corollary}




\begin{example}\label{ex: NoCommonIndexNotMono}
There are extensions of number rings that are locally monogenic but not monogenic.

In \cite{AlpogeBhargavaShnidman}, Alp\"oge, Bhargava, and Shnidman say that an extension $\ZZ_K / \ZZ$ has \emph{no local obstruction to monogenicity} if a local index form represents $1$ over $\mathbb{Z}_p$ for all primes $p$ or $-1$ over $\mathbb{Z}_p$ for all primes $p$. This is a stronger condition than Zariski-local monogenicity, and they show in \cite{AlpogeBhargavaShnidman} and \cite{ABSQuartic} that a positive proportion of quartic and cubic fields are not monogenic despite having no local obstruction to monogenicity.

Narkiewicz \cite[Page 65]{Nark} gives the following concrete example of a locally monogenic but not monogenic extension. Let $L=\QQ(\sqrt[3]{m})$ with $m=ab^2$, $ab$ square-free, $3\dnd m$, and $m\not\equiv \pm1\bmod 9$. The number ring $\ZZ_L$ is not monogenic over $\ZZ$ despite having no common index divisors. We consider the case $ab^2=7\cdot 5^2$ in Example \ref{ex:2genicOverZ_paper2}. This also gives an example of an extension which is Zariski locally monogenic but which has a local obstruction to monogenicity.
\end{example}

\subsection{Monogenicity over geometric points}\label{ssec:Mono_at_geo_points}

\begin{definition}\label{def:mon_over_geo_pts}
Say that $S'$ over $S$ is \emph{monogenic over geometric points} if, for each morphism $\Spec k \to S$ where $k$ is an algebraically closed field, $S' \times_S \Spec k \to \Spec k$ is monogenic.
\end{definition}

While it is a weaker condition than monogenicity over points in general, 
it is equivalent to some conditions that might seem more natural. 

\begin{theorem}\label{thm:monogenic_over_geometric_points}
The following are equivalent:
\begin{enumerate}
    \item\label{geomono1} the local index forms for $S'/S$ are nonzero on each fiber of $\WR_{S'/S} \to S$;
    \item\label{geomono2} for each point $p \in S$ with residue field $k(p)$, there is a finite Galois extension $L / k(p)$ such that $S' \times_S \Spec L \to \Spec L$ is monogenic, and this extension may be chosen to be trivial if $k(p)$ is an infinite field;
    \item\label{geomono3} $S' \to S$ is monogenic over geometric points;
    \item\label{geomono4} there is a jointly surjective collection of maps $\{ U_i \to S \}$ so that $S' \times_S U_i \to U_i$ is monogenic for each $i$;
    \item\label{geomono4.5} $\Gen_{S'/S} \to S$ is surjective;
    \item\label{geomono4.6} $S' \to S$ is \'etale-, smooth-, fppf-, or fpqc-locally monogenic.
\end{enumerate}
If, in addition, $\NGen_{S'/S}$ is a Cartier divisor in $\WR$ (i.e., the local index forms are non-zero divisors), the above are also equivalent to:
\begin{enumerate}[resume]
    \item\label{geomono5} $\NGen_{S'/S} \to S$ (Definition \ref{defn:nonmonogens}) is flat.
\end{enumerate}
\end{theorem}

To see some of the subtleties one can compare item \eqref{geomono1} above with Lemma \ref{prop:mono_over_points_index_form} and item \eqref{geomono4} above with item \eqref{2local} of Theorem \ref{thm:locally_mono_is_mono_over_points}.

\begin{proof}
The assertions are Zariski-local, so we may choose local coordinates as usual ($S = \Spec A$, $S' = \Spec B$, 
coordinates $x_I$).

\eqref{geomono1} $\implies$ \eqref{geomono2}: Suppose first that $p \in S$ is a point with $k(p)$ an infinite field. Let $\mathfrak{p}$ be the corresponding prime of $A$ and write $\ol{\localindex}$ for the restriction of the local index form modulo $\mathfrak{p}$. Recall that since $k(p)$ is infinite, polynomials in $k(p)[x_1,\ldots,x_n]$ are determined by their values on $(x_i) \in k(p)^n$. Since $\ol{\localindex}$ is nonzero, $\ol{\localindex}(a_1, \ldots, a_n)$ must be nonzero for some $(a_1, \ldots, a_n) \in k(p)^n$. This shows that $S' \times_S \Spec k(p) \to \Spec k(p)$ is
monogenic, so we have \eqref{geomono2}.

Next suppose $k(p)$ is a finite field. Then, since $\ol{\localindex}$
is nonzero, there is a finite field extension $L$ (necessarily Galois) of $k(p)$ such that there exists $(a_1, \ldots, a_n) \in L^n$ with $\ol{\localindex}(a_1,\ldots,a_n) \neq 0$. This shows that $S' \times_S \Spec L \to \Spec L$ is
monogenic, so we have \eqref{geomono2} again.

\eqref{geomono2} $\implies$ \eqref{geomono3}: Let $k$ be an algebraically
closed field and $\Spec k \to S$ a map. Let $p$ be the image of $\Spec k$,
and let $L$ be the field extension given by \eqref{geomono2}. Then pullback along $\Spec k \to \Spec L$ implies that $S' \times_S \Spec k \to \Spec k$
is monogenic.

\eqref{geomono3} $\implies$ \eqref{geomono4}: Take $\{ U_i \to S \}$
to be $\{ \Spec \ol{k(p)} \to S \}_{p \in S}$.

\eqref{geomono4} $\implies$ \eqref{geomono1}: For each point $p \in S$,
choose an index $i$ and a point $q_p \in U_i$ mapping to $p$. Let $\localindex$ be an index form around $p$. By pullback,
$S' \times_S q_p \to q_p$ is monogenic, so $\localindex$ pulls back to a nonzero function over $k(q_p)$. Therefore $\localindex$ is nonzero over $k(p)$
as well.

\eqref{geomono2} $\implies$ \eqref{geomono4.5}: For each point $p \in S$, the $\Spec L$ point of $\Gen_{S'/S}$ witnessing monogenicity of $S' \times_S \Spec L \to \Spec L$ is a preimage of $p$.

\eqref{geomono4.5} $\implies$ \eqref{geomono4.6}: Note that $\Gen_{S'/S} \to S$ is smooth, since $\Gen_{S'/S} \to \WR_{S'/S} \to S$ is the composite of an open immersion and an affine bundle. Moreover, the identity function on $\Gen_{S'/S}$ by definition yields a monogenerator for $S' \times_S \Gen_{S'/S} \to \Gen_{S'/S}$. Therefore, $S'/S$ is smooth-locally monogenic. Since the smooth topology is equivalent to the \'etale topology, there is an \'etale cover $U \to S$ factoring through $\Gen_{S'/S} \to S$. Since $S' \times_S \Gen_{S'/S} \to \Gen_{S'/S}$ is already monogenic, $S' \times_S U \to U$ is also monogenic. This \'etale cover is also a cover in the fppf and fpqc topologies.

\eqref{geomono4.6} $\implies$ \eqref{geomono4}: Trivial.


\eqref{geomono1} $\iff$ \eqref{geomono5}: The sequence
\[
  0 \to \indexsheaf_{S'/S} \to \OO_{\WR_{1,S'/S}} \to \OO_{\NGen_{S'/S}} \to 0
\]
may be written as
\[
  0 \to A[x_I] \overset{\localindex}{\to} A[x_I] \to A[x_I]/\localindex \to 0
\]
where $\localindex$ is a local index form for $S'/S$.

Recall that an $A[x_I]$-module $M$ is flat if and only if for each prime $\gp$ of $A$ and ideal $\gq$ of $A[x_I]$ lying
over $\gp$, $M_\gq$ is
flat over $A_\gp$. Therefore, by the local criterion for
flatness\cite[00MK]{sta}, $\NGen_{S'/S}$ is flat over $S$ if and only if
\[
  \Tor_1^{A_\gp}(A_\gp/\gp A_\gp, (A[x_I]/\localindex)_\gq) = 0
\]
for all such ideals $\gp$ and $\gq$.
Therefore, $\NGen_{S'/S}$ is flat if and only if 
\[
  A[x_I]_\gq/\gp A[x_I]_\gq \overset{\localindex (\mathrm{mod} \, \mathfrak{p})}{\longrightarrow} A[x_I]_\gq/\gp A[x_I]_\gq
\]
is injective for all $\gp$ and $\gq$ as above. All of these maps are injective if and only if the maps of $A[x_I]$-modules
\[
  A_{\mathfrak{p}}[x_I]/\mathfrak{p}A_{\mathfrak{p}}[x_I] \overset{\localindex (\mathrm{mod} \, \mathfrak{p})}{\longrightarrow} A_{\mathfrak{p}}[x_I]/\mathfrak{p}A_{\mathfrak{p}}[x_I]
\]
are all injective as $\mathfrak{p}$ varies over the prime ideals of $A$. 
Since $A_{\mathfrak{p}}[x_I]/\mathfrak{p}A_{\mathfrak{p}}[x_I] \cong (A_{\mathfrak{p}}/\mathfrak{p}A_{\mathfrak{p}})[x_I]$ is an integral domain for each $\mathfrak{p}$, injectivity fails if and only if
$\localindex$ reduces to $0$ in the fiber over some $\mathfrak{p}$.
We conclude that \eqref{geomono1} holds if and only if \eqref{geomono5} holds.
\end{proof}

\begin{corollary} \label{cor:inf_residue_fields_mono_at_points}
If all of the points of $S$ have infinite residue fields, then the following are equivalent:
\begin{enumerate}
  \item \label{infresiduemono1} $S' \to S$ is monogenic over geometric points;
  \item \label{infresiduemono2} $S'/S$ is Zariski-locally monogenic.
\end{enumerate}
\end{corollary}

\begin{remark}
The conclusion of Corollary \ref{cor:inf_residue_fields_mono_at_points} fails dramatically if $S$ has finite residue fields.
For $S' \to S$ coming from an extension of number rings condition \eqref{infresiduemono1} \emph{always} holds (see Corollary \ref{cor:number_rings_mono_over_geo_points} below), yet there are extensions that are not locally monogenic.
In this sense, monogenicity is more subtle in the arithmetic context than the geometric one. For an example of an extension that is monogenic over geometric points but is not monogenic over points see Example~\ref{ex:Dedekind_paper2}.
\end{remark}

\subsection{Monogenicity over points}\label{ssec:mon_over_points}

In light of Theorem \ref{thm:locally_mono_is_mono_over_points} and Theorem \ref{thm:monogenic_over_geometric_points}, it is particularly interesting to characterize monogenicity of $S' \to S$ in the case that $S$ is the spectrum of a field $k$. In this case $S'$ is the spectrum
of an $n$-dimensional $k$-algebra $B$. Such algebras are Artinian rings, and a
well-known structure theorem implies that $B$ is a direct product of local Artinian rings $B_i$. We will exploit this to give a complete characterization of Zariski-local monogenicity.

The result in the case that both $S'$ and $S$ are spectra of fields is well-known.

\begin{theorem}[Theorem of the primitive element]
Let $\ell/k$ be a finite field extension. Then $\Spec \ell \to \Spec k$ is monogenic
if and only if there are finitely many intermediate subfields $\ell/\ell'/k$.

In particular, a finite separable extension of fields is monogenic.
\end{theorem}

We next consider the monogenicity of $S' \to S$ when $S'$ is a nilpotent thickening of $\Spec \ell$, leaving $S = \Spec k$ fixed. A key ingredient is a study of the square zero extensions of $\Spec \ell$.

We remark for comparison that the proof below does not consider a nilpotent thickening of the base $S \to \tilde{S}$. In fact, if $S \to \tilde S$ is a nilpotent closed immersion with $S' = S \times_{\tilde S} \tilde S'$, any monogenerator $\theta : S \to \Gen_{S'/S}$ extends to $\tilde S$ locally in the \'etale topology. This results from the smoothness of $\Gen_{\tilde S'/\tilde S} \to \tilde S$. \\

\noindent \textbf{Theorem \ref{thm:local_artin_monogenicity}.} 
Let $S' \to S$ be induced by $k \to B$ where $k$ is a field and $B$ is a local Artinian $k$-algebra with residue field $\ell$ and maximal ideal
$\gm$. Then $S' \to S$ is monogenic if and only if
\begin{enumerate}
  \item $\Spec \ell \to \Spec k$ is monogenic;
  \item $\dim_{\ell} \gm/\gm^2 \leq 1$; and
  \item If $\dim_{\ell} \gm/\gm^2 = 1$ and $\ell/k$ is inseparable, then
  \[
    0 \to \gm/\gm^2 \to B/\gm^2 \to \ell \to 0
  \]
  is a non-split extension.
\end{enumerate}

\begin{proof}
If the tangent space $(\gm/\gm^2)^\vee$ has dimension greater than 1, then no map $S' \to \AA^1_S$ can be injective on tangent vectors as is required for a closed immersion.

On the other hand, if the tangent space of $S'$ has dimension $0$, we have $B = \ell$, and the result is true by hypothesis.

Now suppose the tangent space of $S'$ has dimension $1$. A morphism $S' \to \AA^1_S$ is a closed immersion if and only if it is universally closed, universally injective, and unramified\cite[Tag 04XV]{sta}.

Choose a closed immersion $\Spec \ell \to \AA^1_k$. 
Equivalently, write
$\ell = k[t]/(f(t))$ where $f(t)$ is the monic minimal polynomial of some element $\theta \in \ell$. 
Since $\Spec \ell \to S'$ is a universal homeomorphism\cite[Tag 054M]{sta}, any
extension of $\Spec \ell \to \AA^1_k$ to $S' \to \AA^1_k$ inherits the properties of being universally injective and universally closed from $\Spec \ell \to \AA^1_k$.

Whether such an extension $S' \to \AA^1_k$ is ramified can be checked on the level of tangent vectors\cite[Tag 0B2G]{sta}. It
follows that a morphism $S' \to \AA^1_k$ is a closed immersion if and only if its restriction to the vanishing of $\gm^2$ is. On the other hand, any map $V(\gm^2) \to \AA^1_k$ extends to $S' \to \AA^1_k$ (choose a lift of the image of $t$ arbitrarily).
Therefore, it suffices to consider the case that $S'$ is a square zero extension of $\Spec \ell$.

By hypothesis, we have a presentation of $\ell$ as $k[t]/(f(t))$. We conclude with some elementary deformation theory, see for example \cite[\S 1.1]{sernesi}. We have a square zero extension of $\ell$
\[
  0 \to (f(t))/(f(t)^2) \to k[t]/(f(t))^2 \to \ell \to 0.
\]
By assumption, $B$ is also a square zero extension of $\ell$:
\[
  0 \to \gm \to B \to \ell \to 0.
\]
By \cite[Proposition 1.1.7]{sernesi}, there is a morphism of $k$-algebras $\phi : k[t]/(f(t))^2 \to B$ inducing the identity on $\ell$. Since $(f(t))/(f(t)^2) \cong \ell$ as a $k[t]/(f(t)^2)$ module, 
$\phi$ either restricts
to an isomorphism $f(t)/(f(t))^2 \to \gm$ or else the zero map. In the former
case, the composite $k[t] \to k[t]/(f(t))^2 \to B$ is a surjection, and we are done. In the latter case, $B$ is the pushout of the extension $k[t]/(f(t))^2$ along $(f(t))/(f(t)^2) \overset{0}{\to} \gm$, so $B$ is the split extension $\ell[\epsilon]/\epsilon^2$.

If $\ell/k$ is separable, then the extension $k[t]/(f(t))^2$ is itself split \cite[Proposition B.1, Theorem 1.1.10]{sernesi},
i.e. there an isomorphism $k[t]/(f(t))^2 \cong \ell[\epsilon]/\epsilon^2$. Composing with $k[t] \to k[t]/(f(t)^2)$ gives the required monogenerator.

If $\ell/k$ is inseparable and $B \cong \ell[\epsilon]/\epsilon^2$, we will show that $S' \to S$ is not monogenic. Any generator for $\ell[\epsilon]/\epsilon^2$ over $k$ must also be a generator for $\ell[\epsilon]/\epsilon^2$ over the
maximal separable subextension $k'$ of $\ell/k$, so we may assume
that $\ell/k$ is purely inseparable. Moreover, any generator $\theta$
for $\ell[\epsilon]/\epsilon^2$ over $k$ must reduce modulo $\epsilon$ to a generator $\ol{\theta}$ of $\ell/k$. Since $\ell/k$ is purely inseparable, the minimal polynomial $f(t)$ of $\ol{\theta}$ satisfies $f'(t) = 0$. Note $\theta = \ol{\theta} + c\epsilon$
for some $c \in \ell$.
Since $\theta$ is assumed to be a monogenerator, there is a polynomial $g(t) \in k[t]$ such that $\epsilon = g(\theta)$. Reducing, $\ol{g(\theta)} = g(\ol{\theta}) = 0$, so $g(t) = q(t)f(t)$ for some $q(t) \in k[t]$. Then
\begin{align*}
 g(\theta) &= g(\ol{\theta}) + g'(\theta)c\epsilon \\
   &= 0 + q'(\ol{\theta})f(\ol{\theta})c\epsilon + q(\ol{\theta})f'(\ol{\theta})c\epsilon \\
   &= 0,
\end{align*}
a contradiction. We conclude that in this case $S' \to S$ is not monogenic.
\end{proof}

\begin{remark} \label{rmk:local_artin_mono_special_cases}
In the case that $k$ is perfect, the first and third conditions hold automatically. If $S'$ is regular of dimension 1 the second condition is trivial.
\end{remark}


\noindent \textbf{Theorem \ref{thm:artin_monogenicity}.}
Suppose $S' \to S$ is induced by $k \to A$ where $k$ is a field and $B$ is an Artinian $k$-algebra. Write $B = \prod_i B_i$ where the $B_i$ are local artinian $k$-algebras with respective residue fields $\ell_i$. Then $S' \to S$ is monogenic if and only if
\begin{enumerate}
  \item \label{artinmono1} $\Spec B_i \to S$ is monogenic for each $i$;
  \item \label{artinmono2} for each finite extension $\ell$ of $k$, $S'$ has fewer points with residue field isomorphic to $\ell$ than $\mathbb{A}^1_S$.
\end{enumerate}

\begin{proof}
Note that a map $\bigsqcup_i \Spec B_i \to \AA^1_S$ is a closed immersion if and only
if each map $\Spec B_i \to \AA^1_S$
is a closed immersion and the closed immersions are disjoint: $\Spec B_i \times_{\AA^1_S} \Spec B_j = \varnothing$
for all $i \neq j$. This is equivalent to the statement that
$A[t] \to \prod_{i} B_i$ is surjective if and only if $A[t] \to B_i$ is surjective for each $i$ and $B_i \otimes_{A[t]} B_j = 0$
whenever $i \neq j$, which follows quickly in turn from the Chinese remainder theorem.

The proof of Theorem \ref{thm:local_artin_monogenicity} shows that a closed immersion $\Spec B_i \to \AA^1$ can be chosen with image any of the points of $\AA^1_S$
with residue field $\ell_i$.
Then the condition on numbers of points is exactly what we need for the images of the $\Spec B_i$s not to overlap without running out of points. (Since topologically, the components are single points.)
\end{proof}

\begin{remark}
Condition \eqref{artinmono2} is trivial in the case that the residue fields of $S$ are infinite, highlighting the relative simplicity of monogenicity in the geometric context.

If $S' \to S$ is instead induced by an extension of number rings,
then Remark \ref{rmk:local_artin_mono_special_cases} implies condition \eqref{artinmono1} is trivial. In particular, an extension of $\mathbb{Z}$ has common
index divisors if and only if there is ``too much prime splitting" in the sense
of condition \eqref{artinmono2}. This recovers the theorem of \cite{Hensel1894} (also see \cite[Cor. to Thm. 3]{Pleasants}) that $\gp$ is a common index divisor if and only if there are more primes in $\ZZ_L$ above $\gp $ of residue class degree $f$ than there are irreducible polynomials of degree $f$ in $k(\gp)[x]$ for some positive integer $f.$
\end{remark}

\begin{corollary} \label{cor:number_rings_mono_over_geo_points}
If $S' \to S$ is induced by an extension of number rings $\mathbb{Z}_L/\mathbb{Z}_K$, then $S' \to S$ is monogenic over geometric points.
\end{corollary}
\begin{proof}
Let $p$ be a point of $S$. Let $A = \ZZ_L \otimes_{\ZZ_K} k(p)$ be the ring for the fiber of $S'$ over $\Spec k(p)$. Note that $k(p)$ is either of characteristic 0 or finite, so $k(p)$ is perfect.
Decompose $A$ into a direct product of local Artinian $k(p)$ algebras $A_i$.
Since $k(p)$ is perfect, conditions (1) and (3) of Theorem \ref{thm:local_artin_monogenicity} hold for $\Spec A_i \to \Spec k(p)$. Condition (2) holds as well since $S'$ is regular of dimension 1. Therefore
$\Spec A_i \to \Spec k(p)$ is monogenic for each $i$.

Now consider the base change of $S'$ to the algebraic closure $\ol{k(p)}$ of $k(p)$. Write $B$ for the ring of functions $\ZZ_L \otimes_{\ZZ_K} \ol{k(p)}$ of
this base change, and write $B$ as a product of local Artinian algebras $B_j$. For each $i$ we have that $\Spec A_i \otimes_{k(p)} \ol{k(p)} \to \Spec \ol{k(p)}$ is monogenic. Each $\Spec B_j$ is a closed subscheme of exactly one of the $\Spec A_i \otimes_{k(p)} \ol{k(p)}$s, so by composition, $\Spec B_j \to \Spec \ol{k(p)}$ is monogenic for each $j$. This gives us condition (1) of Theorem \ref{thm:artin_monogenicity} for $S' \times_S \Spec \ol{k(p)} \to \Spec \ol{k(p)}$. Since $\ol{k(p)}$ is infinite, condition (2) holds triviallly. We conclude that $S' \times_S \Spec \ol{k(p)} \to \Spec \ol{k(p)}$ is monogenic, as required.
\end{proof}


\section{\'Etale maps, configuration spaces, and monogenicity}\label{sec:confspsmonetalecase}

This section concerns maps $\pi : S' \to S$ that are \emph{\'etale}, or unramified. Locally, the monogenicity space becomes a configuration space, classifying arrangements of $n$ distinct points on a given topological space. Philosophically, $\Gen_{S'/S}$ therefore regards $S' \to S$ as a twisted family of points to be configured in $\Aff^1$. We are led to interpret $\Gen_{S'/S}$ as an arithmetic
refinement of the configuration space of $\Aff^1$. In Remark~\ref{rem:braidgroupgrothteichetalefundgp}, we see that an action of the absolute Galois group $\text{Gal}(\overline{\QQ}/\QQ)$ on the \'etale fundamental group of $\Gen_{S'/S}$ has been observed in anabelian geometry. We end Subsection \ref{ssec:etalecase} with a handful of exotic applications in other areas. 

All extensions $S' \to S$ sit somewhere between the \'etale case and jet spaces $\Spec A[\epsilon]/\epsilon^{n} \to \Spec A$ (see \cite[Example 2.7, 4.3, 4.16]{abhs_paper_1}), between being totally unramified and totally ramified. 
In Section \ref{ssec:whenisetale}, we recall a general construction of the discriminant which cuts out the locus of ramification. 
Specifically, \cite[\S 6]{poonenmodspoffiniteflat} says that our description in the \'etale case holds precisely away from the vanishing of the discriminant. The discriminant plays a similar role in the classical case when investigating the monogenicity of an extension defined by a polynomial. We end with some remarks on using stacks to promote a ramified cover of curves to an \'etale cover of stacky curves as in \cite{costello}.

\subsection{The trivial cover}\label{ssec:trivialcover}

We start with the simplest case of an \'etale cover: the trivial cover of
$S$ by several copies of itself. We work rather concretely and revisit the
general situation with more sophistication in the next subsection.
Write $\num{n} = \{1, 2, \dots, n\}$ and $\num{n}_S = \bigsqcup_{1 \leq i \leq n} S \to S$ for the trivial degree $n$ finite \'etale cover.

\begin{example}[Monogenicity of a trivial cover] \label{ex:mono_of_trivial_cover}
Let $S' = \num{n}_S$ and let $\pi : \num{n}_S \to S$ be the
map induced by the identity on each copy of $S$. Given a commutative diagram
\[
\begin{tikzcd}
    & \mathbb{A}^1_S \ar[d] \\
    \langle n\rangle_S \ar[ur, "\sqcup_{i} f_i"] \ar[r, "\pi"] & S,
\end{tikzcd}
\]
one expects that the map $\sqcup_{i} f_i$ will be a closed immersion
if and only if $f_i(s) \neq f_j(s)$ for all $1 \leq i < j \leq n$ and $s \in S$. A computation in coordinates will confirm this.

We will use the notation of Definition \ref{def:local_coords} and Theorem \ref{thm:gen_local_coords} to compute $\Gen_{S'/S}$ in local coordinates.
Working Zariski locally on $S$, we may assume that $S' = \Spec A^n$, $S = \Spec A$, and that $e_1, \ldots, e_n$ are the standard basis vectors for $A^n$. 
Let $x_1,\ldots, x_n$ be the corresponding coordinates for $\WR_{S'/S}$, so that
\[
  \WR_{S'/S} \cong \AA^n_S.
\]
This isomorphism identifies the $T$-point $(x_1,\ldots,x_n)$ of $\AA^n_S$
with the map $\sqcup_{i} x_i : \langle n \rangle_T \to \AA^1_T$ whose restriction to the $i$th copy of $T$ in $\langle n \rangle_T$ is $x_i$.

Next, observe that in $\mathscr{O}_{\WR_{S'/S}}(S)$
\[
  (x_1e_1 + \cdots + x_ne_n)^j = x_1^je_1 + \cdots + x_n^je_n
\]
for all $0 \leq j \leq n - 1$. Therefore, the matrix of coefficients
$M(e_1, \ldots, e_n)$ is the Vandermonde matrix with $i$th column given by $\begin{bmatrix} 1 & x_i & x_i^2 & \cdots & x_i^{n - 1}\end{bmatrix}^{T}$.
The index form is then the well-known Vandermonde determinant:
\[
  \localindex(e_1,\ldots, e_n)(x_1,\ldots, x_n)  = \det M(e_1, \ldots, e_n) = \pm\prod_{i < j} (x_i - x_j).
\]
The index form vanishes therefore on the so-called \emph{fat diagonal} $\hat{\Delta} \subseteq \AA^n$, given by the union of all loci $V(x_i - x_j)$ where two coordinates are equal.

It follows that
\[
  \Gen_{S'/S} \cong \SSpec{S}{\OO_S\left[x_1,\ldots,x_n, \prod_{i < j}(x_i - x_j)^{-1}\right]},
\]
the complement in $\AA^n_S$ of $\hat{\Delta}$. This space is otherwise known
as the space of \emph{ordered configurations} of $n$ points, $\Conf_n(\AA^1_S) \to S$.
\end{example}

Slightly more abstract reasoning yields a similar result if $X$ is any quasi-projective $S$ scheme.

\begin{example}[$\Gen_{X,S'/S}$ for a trivial cover]
Let $X \to S$ be a quasiprojective map, $S' = \num{n}_S$ and let $\pi : \num{n}_S \to S$
be the map induced by the identity on each copy of $S$. Observe that if $T \to S$ is an $S$-scheme, we have natural
identifications
\begin{align*}
  \WR_{X,S'/S}(T) &\cong \left\{ \text{ maps of $T$-schemes }f : \langle n \rangle_T \to X \times_S T \right\} \\
    &\cong \left\{ \text{ $n$-tuples of maps of $T$-schemes }f_i : T \to X \times_S T \text{, where }i = 1, \ldots, n\right\} \\
    &\cong \left\{ \text{ $n$-tuples of maps of $S$-schemes }f_i : T \to X \text{, where }i = 1, \ldots, n\right\} \\
    &\cong \underbrace{X\times_S X\times_S \cdots \times_S X}_{n\text{-times}}(T) = X^{\times_S^n}(T),
\end{align*}
so we may identify $\WR_{X, S'/S}$ with the $n$-fold fiber product of $X$ over $S$.

For each $1 \leq i < j \leq n$, we can construct a subscheme $\Delta_{i,j}$ of $X^{\times_S^n}$ consisting of the
points whose $i$th and $j$ coordinates are equal. We let the fat diagonal $\hat\Delta$ be the scheme theoretic union
of the subschemes $\Delta_{i,j}$. Since $X \to S$ is separated, this fat diagonal is a closed subscheme of $X^{\times_S^n}$.

Observe that any morphism of $T$-schemes $f = \bigsqcup_{i = 1}^n f_i : \num{n}_T \to X \times_S T$ in $\WR_{X,S'/S}(T)$
will be proper and unramified as $\pi : \num{n}_T \to T$ is proper and unramified and $X \times_S T \to T$ is separated. In addition, for each point $x \in \num{n}_T$, the induced field extension $\kappa(x) \supseteq \kappa(f(x))$ is an isomorphism, since the same is true of the map $\pi : \num{n}_T \to T$. 

Now, by \cite[Tag 01S4, (2)$\iff$(3) ]{sta} and \cite[Tag 04XV, (1)$\iff$(3)]{sta}, $f$ is a closed immersion if and only if it is injective. This happens if and only if the corresponding function $\prod_{i = 1}^n f_i : T \to X^{\times_S^n}$ factors through the complement of the fat diagonal. Therefore $\Gen_{X, S'/S} \cong X^{\times_S^n} - \hat{\Delta}.$
\end{example}

\subsection{The case of \'etale $S' \to S$}\label{ssec:etalecase}

Consider the category $\Sch{\ast}$ of schemes over a final scheme $\ast$ equipped with the \'etale topology. For example, take $\ast = \Spec \ZZ$ or $\Spec \CC$. Write $\Sigma_n$ for the symmetric group on $n$ letters and $B\Sigma_n$ for the stack on $\Sch{\ast}$ of \'etale $\Sigma_n$-torsors. 

Regard $\Sigma_{n-1}$ as the subgroup of $\Sigma_n$ of permutations fixing the $n$th letter,
and let $B\Sigma_{n-1} \to B\Sigma_n$ be the map induced by the inclusion. The isomorphism class of the resulting map of classifying spaces is unchanged if $\Sigma_{n-1}$ is taken as the subgroup fixing some other letter, since resulting inclusion map only differs from this one by conjugation. The morphism $B\Sigma_{n-1} \to B\Sigma_n$ is the universal $n$-sheeted cover in the following sense.

\begin{lemma}[{\cite[Lemma 2.2.1]{costello}, \cite[Lemma 3.2]{mycostellocorrection}}]
Let $n$ be a positive integer. Let $\mathcal{C}$ be the fibered category over $\Sch{\ast}$ with:
\begin{enumerate}
    \item objects the finite \'etale morphisms $\pi : S' \to S$ of degree $n$;
    \item arrows the cartesion diagrams
    \[
    \begin{tikzcd}
        T' \ar[r] \ar[d] & S' \ar[d] \\
        T \ar[r] & S;
    \end{tikzcd}
    \]
    \item projection to $\Sch{\ast}$ given by $S' \to S \mapsto S$.
\end{enumerate}
Then there is an equivalence of fibered categories $B\Sigma_n \to \mathcal{C}$ given by taking a map $f : S \to B\Sigma_n$ to
the pullback of $B\Sigma_{n-1} \to B\Sigma_n$ along $f$.
\end{lemma}

Recall that pullback squares of schemes 
\[\begin{tikzcd}
T' \ar[r] \ar[d] \pb      &S' \ar[d]         \\
T \ar[r]       &S
\end{tikzcd}\]
induce identifications 
\[\WR_{T'/T} \simeq \WR_{S'/S} \times_S T  \quad \ \ \ \text{ and } \ \ \  \quad \Gen_{1, T'/T} \simeq \Gen_{1, S'/S} \times_S T.\]
Reduce thereby to the universal $n$-sheeted finite \'etale cover $S' = B\Sigma_{n-1}$, $S = B\Sigma_n$. 
Each has an affine line $\Aff^1_{B\Sigma_n} = [\Aff^1/\Sigma_n]$ obtained via quotienting by the trivial action.\\

\noindent \textbf{Theorem \ref{thm:three_isoms}.}
There are isomorphisms
\begin{align*}
  \WR_{B\Sigma_{n-1}/B\Sigma_n} &\cong [\AA^n / \Sigma_n], \\
  \NGen_{B\Sigma_{n-1}/B\Sigma_n}  &\cong [\hat{\Delta} / \Sigma_n],  \\
  \text{ and } \Gen_{B\Sigma_{n-1}/B\Sigma_n} &\cong [(\AA^n - \hat{\Delta}) / \Sigma_n]
\end{align*}
in which the action by $\Sigma_n$ is in each case the appropriate restriction of the permutation action on coordinates of $\AA^n$.

\begin{proof}
We observe that the $n$-sheeted cover associated to the trivial torsor $* \to B\Sigma_n$ is the trivial cover $\langle n \rangle \to *$. Therefore, by our work in the case of a trivial cover
\begin{align*}
  \WR_{B\Sigma_{n-1}/B\Sigma_n} \times_{B\Sigma_n} * &\cong \WR_{\langle n \rangle / *} \cong \AA^n \\
  \NGen_{B\Sigma_{n-1}/B\Sigma_n} \times_{B\Sigma_n} * &\cong \NGen_{\langle n \rangle / *} \cong \hat{\Delta} \\
  \Gen_{B\Sigma_{n-1}/B\Sigma_n} \times_{B\Sigma_n} * &\cong \Gen_{\langle n \rangle / *} \cong \AA^n - \hat{\Delta} \cong \Conf_n(\AA^1).
\end{align*}

There is a $\Sigma_n$ action on $\WR_{\langle n \rangle / *}$ so that $\WR_{B\Sigma_{n-1} / B\Sigma_n}$ is the stack quotient of $\WR_{\langle n \rangle / *}$ by $\Sigma_n$. Pulling back $\WR_{\langle n \rangle / *} \to *$ to $* \times_{B\Sigma_n} *$ in both ways shows that the action is given by permuting the sheets of $\langle n \rangle$. Under the isomorphism of $\WR_{\langle n \rangle / *}$ with $\AA^n$ of Example \ref{ex:mono_of_trivial_cover}, the action is given by permuting the coordinates.

We conclude
\begin{align*}
  \WR_{B\Sigma_{n-1}/B\Sigma_n} &\cong [\AA^n / \Sigma_n] \\
  \NGen_{B\Sigma_{n-1}/B\Sigma_n}  &\cong [\hat{\Delta} / \Sigma_n] \\
  \Gen_{B\Sigma_{n-1}/B\Sigma_n} &\cong [(\AA^n - \hat{\Delta}) / \Sigma_n]
\end{align*}
in which the action by $\Sigma_n$ is in each case the appropriate restriction of the permutation action on coordinates of $\AA^n$. 
\end{proof}

The space $\Gen_{B\Sigma_{n-1}/B\Sigma_n}$ is also interpretable as
the space of \emph{unordered configurations of points}:
\[
  \Gen_{B\Sigma_{n-1}/B\Sigma_n} \cong \UConf_n(\AA^1) = \{ (x_1, \ldots, x_n) \mid x_i \neq x_j \} / \Sigma_n.
\]

Observe that the fat diagonal $\hat{\Delta}$ is exactly the locus of $\Aff^n$ where $\Sigma_n$ has stabilizers. The coarse moduli space of $\stquot{\Aff^n/\Sigma_n}$ is $\Aff^n$ by the fundamental theorem of symmetric functions,
with the composite
\[\Aff^n \to \stquot{\Aff^n/\Sigma_n} \to \Aff^n; \quad \quad \quad \overline{x} = (x_1, \dots, x_n) \mapsto (s_1(\overline{x}), s_2(\overline{x}), \dots, s_n(\overline{x}))\]
given by the elementary symmetric polynomials $s_i(x_1, \dots, x_n)$ \cite[\S 16.1-2]{artinalgebra}. The composite sends a list of $n$ roots to the coefficients of the monic polynomial of degree $n$ vanishing at those roots, up to sign:
\[(t-x_1) \cdots (t-x_n) = t^n - s_1(\overline{x})t^{n-1} + s_2(\overline{x})t^{n-2} - \cdots \pm s_n(\overline{x}).\]
The assignment is plainly invariant under relabeling the $x_i$ by $\Sigma_n$. 

The map to the coarse moduli space $\stquot{\Aff^n/\Sigma_n} \to \Aff^n$ is an isomorphism precisely over $\Gen_{1, B\Sigma_{n-1}/B\Sigma_n}$. The image of $\NGen_{B\Sigma_{n-1}/B\Sigma_n}$ in $\Aff^n$ is the closed subscheme cut out by the discriminant of the above polynomial
\begin{align*}
\disc\left(\prod_{i=1}^n\left(t-x_i\right)\right) &=D(s_1(\overline{x}), s_2(\overline{x}), \dots, s_n(\overline{x}))   \\
        &=\prod_{i < j} (x_i - x_j)^2,
\end{align*}
the square of the Vandermonde determinant. The resulting divisor is the pushforward of $\NGen_{B\Sigma_{n-1}/B\Sigma_n}$ to the coarse moduli space $\Aff^n$.

We summarize the above discussion for general targets $X$ in the place of $\Aff^1$:

\begin{theorem}\label{thm:etgenisconf}

Let $X$ be a quasiprojective scheme, and let $X_{B\Sigma_n} \coloneqq [X / \Sigma_n] = X \times B \Sigma_n$ be the stack quotient by the trivial $\Sigma_n$ action. 
\begin{itemize}
    \item The Weil Restriction is the stacky symmetric product:
    \[\WR_{[X / \Sigma_n], B\Sigma_{n-1}/B\Sigma_n} \coloneqq   \stquot{\Sym n X} = \stquot{X^n/\Sigma_n}.\]
    
    \item The space of monogenerators for $S' = B\Sigma_{n-1}$, $S = B\Sigma_n$ is the $n$th unordered configuration space:
    \[\Gen_{[X / \Sigma_n], B\Sigma_{n-1}/B\Sigma_n} = \UConf_n X \coloneqq   \{(x_1, \dots, x_n) \, | \, x_i \neq x_j \text{ for }i \neq j\}/\Sigma_n.\]
    
    \item The complementary space of non-monogenerators is the stack quotient by $\Sigma_n$ of the ``fat diagonal'' of $n$ points in $X$ which are not pairwise distinct:
    \[\NGen_{[X / \Sigma_n], B\Sigma_{n-1}/B\Sigma_n} = \stquot{\hat{\Delta}_X/\Sigma_n} = \{(x_1, \dots, x_n) \, | \, \text{some }x_i = x_j, i \neq j\}/\Sigma_n.\]
\end{itemize}

\end{theorem}

\subsection{Implications}

The rest of the section gives sample applications, exotic examples, and directions based on the correspondence with configuration spaces. 

\begin{remark}\label{rem:braidgroupgrothteichetalefundgp}
Classical work on the analogues of monogenicity in complex geometry, such as \cite{hansen_polynomials}, has recognized that embeddings into $\AA^1$-bundles are closely related to the braid group, essentially because the fundamental group of the configuration space of $n$ points in $\CC$ is the braid group on $n$ strands. In the scheme theoretic setting, our best analogue of the fundamental group is the \'etale fundamental group.

The computations above imply that
\[
  \Gen_{1,B\Sigma_{n-1}/B\Sigma_n} \times_{\ZZ} \QQ \cong \UConf(\mathbb{A}^1_\QQ).
\]
In \cite{galoisactiononbraidgroup}, it is computed that this space has \'etale fundamental group a semi-direct product
\[
  \hat{B_{n}} \rtimes G_{\QQ},
\]
where $\hat{B_n}$ is the profinite completion of the braid group on $n$ strands and $G_{\QQ}$ is the absolute Galois group of $\QQ$. As discussed in \cite{galoisactiononbraidgroup}, the conjugation action of $G_\QQ$ on $\hat{B_n}$ extends to an action by the Grothendieck-Teichm\"uller group $\hat{GT}$. Conjecturally, $G_\QQ=\hat{GT}$. Though all varieties over $\QQ$ yield actions of the Galois group $G_\QQ$, we were surprised to rediscover one of its central representations used in number theory. 

\end{remark}


The following result is well-known, as the square of the Steinitz class is the discriminant, and the discriminant is a unit when $S' \to S$ is \'etale. However, we have a pleasant alternative proof in terms of our universal \'etale cover.

\begin{theorem} \label{thm:square_zero_steinitz}
If $S' \to S$ is \'etale, the Steinitz class is 2-torsion in $\Pic(S)$. If $S$ has characteristic $2$, the Steinitz class vanishes.
\end{theorem}
\begin{proof}
It is enough to show that the Steinitz class $\det \pi_*\mathscr{O}_{S'}$ is 2-torsion for the universal case $\pi : B\Sigma_{n-1}\to B\Sigma_n$. Consider the pullback square
\[
\begin{tikzcd}
  \langle n \rangle \ar[r, "j"] \ar[d, "\tau"] & B\Sigma_{n-1} \ar[d, "\pi"] \\
  * \ar[r, "i"] & B\Sigma_{n}.
\end{tikzcd}
\]

The pushforward $\pi_*\mathscr{O}_{B\Sigma_{n-1}}$ is trivialized on the \'etale cover $i : * \to B\Sigma_n$, as
\[
  i^*(\pi_*\mathscr{O}_{B\Sigma_{n-1}}) \cong \tau_*\mathscr{O}_{\langle n \rangle} = \mathscr{O}_*^n.
\]

We find that the descent datum for $\pi_*\mathscr{O}_{B\Sigma_{n-1}}$ with respect to this cover has gluing $\mathscr{O}_{* \times \Sigma_n}^n \to \mathscr{O}_{* \times \Sigma_n}^n$ on $* \times_{B\Sigma_n} * \cong * \times \Sigma_n$ given by permuting the coordinates by $\sigma$ over $* \times \sigma$ for each $\sigma \in \Sigma_n$. This is represented by a permutation matrix, which has determinant $\pm 1$. Therefore the gluing data for $\det \pi_*\mathscr{O}_{B\Sigma_{n-1}}$ is given locally by multiplying by $\pm 1$. Since $\pm 1$ is 2-torsion in $\mathscr{O}^*$ and trivial if $S$ has characteristic 2, the result follows.
\end{proof}

\begin{example}[Torsors for finite groups]\label{ex:fintorsors}

Let $G$ be a finite group. A $G$-torsor $S' \to S$ is, in particular, a finite \'etale map of degree $n = \# G$ admitting the above description. Notice that the action of $G$ on $S'$ induces an action of $G$ on $\Gen_{S'/S}$.

The map $S' \to S$ is classified by a map $S \to BG$, the stack of $G$-torsors, and we may regard $\Gen_{S'/S}$ as pulled back from either the monogenicity space of the universal $G$-torsor, $\Gen_{*/BG}$, or the monogenicity space of the universal $n$-fold cover $\Gen_{B\Sigma_{n-1}/B\Sigma_n}$. To compare the two, observe that the left regular representation $G \action G$ gives an inclusion $G \subseteq \Sigma_n$ upon ordering the set $G$. The induced representable map $BG \to B\Sigma_n$ is essentially independent of the ordering since different orderings induce conjugate maps. The classifying map $S \to B\Sigma_n$ is the composite $S \to BG \to B\Sigma_n$ with the left regular representation. The monogenicity space $\Gen_{*/BG}$ is $[\AA^{|G|} - \hat{\Delta} / G ]$ where $G$ acts on $\AA^{|G|}$ by permuting the basis vectors by the left regular representation.

A similar description locally holds for other finite \'etale group schemes. For merely finite flat group schemes $G$ such as $\alpha_p, \mu_p$ in characteristic $p$, the group action on the monogenicity space of $G$-torsors $S' \to S$ still holds but the local decomposition $S' = \bigsqcup_n S$ and $\Sigma_n$ action do not. 

\end{example}

\begin{corollary}
If $S' \to S$ is a $G$-torsor for $G$ a finite group, and either
\begin{enumerate}
  \item $|G|$ is odd
  \item $|G|$ is even and $G$ has non-cyclic Sylow 2-subgroup
\end{enumerate}
then the Steinitz class of $S' \to S$ is trivial in $\Pic(S)$
\end{corollary}
\begin{proof}
Repeating the construction of Theorem \ref{thm:square_zero_steinitz},
we see that if the left regular representation of $G$ factors through $A_n$,
the Steinitz class is trivial. The conditions given identify precisely when this happens.
\end{proof}

The stacks $\WR_{1, B\Sigma_{n-1}/B\Sigma_n}$ we study arise naturally in log geometry as ``Artin fans'' \cite{boundednessartinfans}.

\begin{example}[Moduli spaces of curves in genus 0]\label{ex:modgenus0curvesgen}

Let $M_{0, n}$ be the moduli stack of \textit{smooth} curves of genus 0 (i.e. $\PP^1$) with $n$ marked points. The evident \cite{westerlandconfspaces} isomorphism with a quotient of configuration space gives:
\[\stquot{\Gen_{\PP^1, \CC^n/\CC}/\PGL_2} \simeq \stquot{\Conf_n(\PP^1)/\PGL_2} \simeq M_{0, n}.\]
One can always put the first point at $\infty$ and get equivalent descriptions:
\[M_{0, n} \simeq \stquot{\Gen_{\Aff^1, \CC^n/\CC}/\Afftrans^1} \simeq \stquot{\Conf_n(\CC)/\Afftrans^1},\]
where $\Afftrans^1$ is the group of affine transformations $\GG_m \ltimes \Aff^1$. The stack quotient classifies local affine equivalence classes of monogenerators, as detailed in Section \ref{sec:twistedmon}. 

One can likewise obtain the other moduli spaces of curves by an ad hoc construction. Consider $\frak U \to \frak M$ the universal connected, proper, genus-$g$ nodal curve, its relative smooth locus $\frak U^{sm} \subseteq \frak U$, and the monogenicity space
\[\Gen_{\frak U^{sm}, \num{n}_{\frak M}/\frak M}\]
of the trivial cover $\num{n}$ over the moduli space $\frak M$. The monogenicity stack is naturally isomorphic to the space of nodal, $n$-marked curves $\frak M_{g, n}$. One can also obtain the open substack of \emph{stable} curves as the universal Deligne-Mumford locus $\overline{\mathscr{M}}_{g, n} \subseteq \frak M_{g, n}$.
\end{example}

\subsection{When is a map \'etale?}\label{ssec:whenisetale}

We recall from \cite[\S 6]{poonenmodspoffiniteflat} that a map $S' \to S$ is \'etale precisely when the discriminant of the algebra does not vanish. We recall from \cite{poonenmodspoffiniteflat} that there is an algebraic moduli stack $\frak A_n$ of finite locally free algebras and the affine scheme of finite type $\frak B_n$ parametrizing such algebras together with a choice of global basis $\cal Q \simeq \bigoplus \OO_S \cdot e_i$.

Suppose $\pi : S' \to S$ comes from a finite flat algebra $\cal Q$ with a global basis $\varphi : \cal Q \simeq \bigoplus_{i=1}^n \OO_S \cdot e_i$, corresponding to a map $S \to \frak B_n$. There is a trace pairing $\text{Tr} : \cal Q \to \OO_S$ 
\cite[0BSY]{sta} which we can use to define the discriminant:
\[\disc(\cal Q, \varphi) \coloneqq   \det \stquot{\text{Tr}(e_i e_j)} \in \Gamma(\OO_S)\]

Changing $\varphi$ changes the function $\disc$ by a unit. The function $\disc$ does not descend to $\frak A_n$, but the vanishing locus $V(\disc) \subseteq \frak B_n$ does. Writing $\frak B_n^{et}$, $\frak A_n^{et}$ for the open complements of the vanishing locus $V(\disc)$, a map $\pi : S' \to S$ is \'etale if and only if $S \to \frak A_n$ factors through the open substack $\frak A_n^{et} \subseteq \frak A_n$ \cite[Proposition 6.1]{poonenmodspoffiniteflat}.

\begin{remark}

Most finite flat algebras are \textit{not} \'etale, nor are they degenerations of \'etale algebras. B. Poonen shows the moduli of \'etale algebras inside of all finite flat algebras $\frak A_n^{et} \subseteq \frak A_n$ cannot be dense by computing dimensions \cite[Remark 6.11]{poonenmodspoffiniteflat}. The closure $\overline{\frak A}_n^{et}$ is nevertheless an irreducible component.

\end{remark}

What if $S' \to S$ is not \'etale? Readers familiar with \cite{costello} know one can sometimes endow a ramified map $S' \to S$ with stack structure $\tilde S$ and $\tilde S'$ at the ramification to make $\tilde S' \to \tilde S$ \'etale. Then all $\tilde S' \to \tilde S$ are $\Sigma_n$-torsors, and not just ramified covers $S' \to S$. The ideas in Section \ref{ssec:etalecase}, in particular an analogue of Theorem~\ref{thm:etgenisconf}, apply in this level of generality.  We sketch these ideas over $\CC$.

Consider 
\begin{equation*}
y^2 = x (x-1) (x - \lambda),
\end{equation*}
for some $\lambda \in \CC$. If $C \coloneqq \PP^1_{\CC}$ and $C'$ is the projective closure of the above affine equation, the projection $(x,y) \mapsto x$ extends to a  finite locally free map $\pi : C' \to C$. This is in Situation \ref{sit:gensetup_paper2} so our definitions make sense for it. However $\pi$ is ramified at four points, preventing us from interpreting its monogenicity space using the perspective of this section. Nevertheless, we may observe that the function $y$ gives a section of $\Gen_{1, C'/C}$ over $C \setminus \infty$. The section naturally extends to a section of $\WR_{\PP^1, C'/C}$ over all of $C$.

Let $X \coloneqq \PP^1_C$. If we work over $\CC$ and endow $C'$ and $C$ with stack structure to obtain a finite \emph{\'etale} cover of stacky curves $\tilde{C}' \to \tilde{C}$ as in \cite{costello}, the stacky finite \'etale cover together with the map $\tilde C' \to X$ is parameterized by a representable map $\tilde{C} \to [\Sym n {X}]$ to the stack quotient
\[[\Sym n {X}] \coloneqq  \stquot{X^n/\Sigma_n}.\]
We can similarly allow $C'$ and $C$ to be nodal families of curves over some base $S$. Maps from nodal curves $\tilde C$ over $S$ entail an $S$-point of the moduli stack $\frak M ([\Sym n {X}])$ of prestable maps to the symmetric product. As in \cite[Proposition 2.3]{abhs_paper_1}, there is an open substack for which the map from the coarse space $C'$ to $X$ is a closed immersion. The stack $\frak M([\Sym n {X}])$ splits into components indexed by the ramification profiles of the cover of coarse spaces $C' \to C$.

There are some subtleties in characteristic $p$---one cannot treat all ramification as a $\mu_n$ torsor because some ramification is a $\ZZ/p\ZZ$-torsor in characteristic $p$. The formalism of tuning stacks \cite{ellenbergsatrianodzb} is a substitute in arbitrary characteristic.


\section{Twisted monogenicity}\label{sec:twistedmon}

The \emph{Hasse local-to-global principle} is the idea that ``local'' solutions to a polynomial equation over all the $p$-adic fields $\QQ_p$ and the real field $\RR$ can piece together to a single ``global'' solution over $\QQ$. We ask the same for monogenicity: given local monogenerators, say over completions or local in the Zariski or \'etale topologies, do they piece together to a single global monogenerator?

The Hasse principle fails for elliptic curves. Let $E$ be an elliptic curve over a number field $K$ and consider all its places $\nu$. The Shafarevich-Tate group $\Sha(E/K)$ of an elliptic curve sits in an exact sequence 
\[0 \to \Sha(E/K) \to H^1_{et}(K, E) \to \prod_{\nu} H^1_{et}(K_\nu, E).\]
Elements of $\Sha$ are genus-one curves with rational points over each completion $K_\nu$ that do \emph{not} have a point over $K$. Similarly, we want sequences of cohomology groups to control when local monogenerators do or do not come from a global monogenerator.

For such a sequence, one needs to know how a pair of local monogenerators can differ. One would like a group $G$ or sheaf of groups transitively acting on the set of local monogenerators so that cohomology groups can record the struggle to patch local monogenerators together into a global monogenerator.

Suppose $B/A$ is an algebra extension inducing $S' \to S$ and $\theta_1, \theta_2 \in B$ are both monogenerators. Then 
\[\theta_1 \in B = A[\theta_2],\quad \quad \quad \theta_2 \in B = A[\theta_1],\]
so each monogenerator is a polynomial in the other:
\[
\theta_1 = p_1(\theta_2) \quad \text{ and } \quad \theta_2 = p_2(\theta_1), \quad \text{ with } \quad p_1(x), p_2(x) \in A[x].
\]
We can think of the $p_i(x)$ as transition functions or endomorphisms of the affine line $\Aff^1$. Even though $p_1(p_2(\theta_1)) = \theta_1$, it is doubtful that $p_1 \circ p_2 = \id_{\AA^1}$ or even that $p_i(x)$ are automorphisms of $\Aff^1$. 

One might attempt to find a group $G$ containing all possible polynomials $p_1(x), p_2(x)$. We would then have a homomorphism (of non-commutative monoids) $E \to G$
where $E$ is some sub-monoid of $\End(\AA^1)$, the monoid of endomorphisms of $\AA^1$ (equivalently, the monoid of one-variable polynomials under composition). Even if we only insist that $E$ contains $x$, $-x$, and $x^2$, we find that the images of $x$
and $-x$ coincide in $G$ since both compose with $x^2$ to the same polynomial. This is not acceptable as $x$ and $-x$ act in distinct
ways on monogenerators.

Instead of working with the group of all possible polynomial transition functions as above, we \emph{require} our transition functions $p_i(x)$ to lie in a group $G \action \Aff^1$ acting on $\Aff^1$. Two particularly natural options for $G$ present themeselves, namely the group sheaves:
\[\GG_m(A) = A^\ast, \quad \quad \quad u \cdot u' \coloneqq   uu'\]
\[\Afftrans^1(A) = A^\ast \ltimes A, \quad \quad \quad (u, v)\cdot(u', v')\coloneqq   (uu', uv' + v).\]
Affine transformations $\Afftrans^1$ are essentially polynomials $ux + v$ of degree one under composition. These act on monogenerators:
\begin{alignat*}{3}
    \GG_m \action \Aff^1 &:  \quad \quad \quad \quad&a \in A^\ast, \theta \in \Gen(A),   \quad \quad \quad \quad&a.\theta \coloneqq   a\cdot \theta, \\
\Afftrans^1 \action \Aff^1 &:  \quad \quad \quad \quad &a \in A^\ast, b \in A, \theta \in \Gen(A),  \quad \quad \quad \quad &(a, b).\theta \coloneqq   a\theta + b.
\end{alignat*}

\begin{definition}[Twisted Monogenerators]\label{defn:twistedmon}

A ($\GG_m$)-twisted monogenerator for $B/A$ is:
\begin{enumerate}
    \item a Zariski open cover $\Spec A = \bigcup_i D(f_i)$ for elements $f_i \in A$,
    \item a system of ``local'' monogenerators $\theta_i \in B\left[f_i^{-1}\right]$ for $B[f_i^{-1}]$ over $A[f_i^{-1}]$, and
    \item\label{eitem:unitstwistedmon} units $a_{ij} \in A\left[f_i^{-1}, f_j^{-1}\right]^\ast$
\end{enumerate}
such that
\begin{itemize}
    \item for all $i,j$, we have $a_{ij} . \theta_j = \theta_i,$
    \vspace{.1 cm}
    \item for all $i,j,k$, the ``cocycle condition'' holds:
\[a_{ij} . a_{jk} = a_{ik}.\]
\end{itemize}

Two such systems $\{(a_{ij}), (\theta_i)\}$, $\{(a_{ij}'), (\theta_i')\}$ are \emph{equivalent} if they differ by further refining the cover $\Spec A = \bigcup D(f_i)$ or  global units $u \in A^\ast$: $u\cdot a_{ij} = a'_{ij}$, $u \cdot \theta_i = \theta_i'$. 

Likewise $B/A$ is $\Afftrans^1$-\emph{twisted monogenic} if there is a cover with $\theta_i$'s as above, but with units \eqref{eitem:unitstwistedmon} replaced by pairs $a_{ij}, b_{ij} \in A \left[\frac{1}{f_i}, \frac{1}{f_j}\right]$ such that each $a_{ij}$ is a unit and $a_{ij} \theta_j + b_{ij} = \theta_i$. 

\end{definition}

The elements $\theta_i$ may or may not come from a single global monogenerator $\theta \in A$. Nevertheless, the transition functions $(a_{ij})$ or $(a_{ij}, b_{ij})$ 
define an affine bundle $L$ on $\Spec A$ with global section $\theta$ induced by the $\theta_i$'s. We say $S'/S$ is ``twisted monogenic'' to mean there exists a $\GG_m$-twisted monogenerator and similarly say ``$\Afftrans^1$-twisted'' monogenic. Both are clearly Zariski-locally monogenic. 

Compare with Cartier divisors: 
\begin{center}
\begin{tabular}{c|c}
twisted monogenerator &Cartier divisor  \\
global monogenerator &rational function   \\
$\GG_m$ / $\Afftrans^1$ action &differing by units.
\end{tabular}
\end{center}

We recall the notions of ``multiply monogenic orders'' and ``affine equivalence'' in the literature. 
Two monogenerators $\theta_1, \theta_2 \in \Gamma(\Spec A, \Gen_{B/A})$ are said to be ``affine equivalent'' if there are $u \in A^\ast$, $v \in A$ such that $u\theta_1 + v = \theta_2$. In other words, affine equivalence classes are elements of the quotient $\Gamma(\Spec A, \Gen_{B/A})/\Afftrans^1(A)$. Under certain hypotheses in Remark \ref{rmk:affineequivvstwistedmon}, $\Afftrans^1$-twisted monogenicity is parameterized by the sheaf quotient $\Gen_{B/A}/\Afftrans^1$. There is almost an ``exact sequence''
\[\Afftrans^1(A) \to \Gamma(\Gen_{B/A}) \to \Gamma(\Gen_{B/A}/\Afftrans^1) \to H^1(\Afftrans^1)\]
that dictates whether a twisted monogenerator comes from an affine equivalence class of global monogenerators. 

We warm up with a classical approach to $\GG_m$-quotients, namely taking $\Proj$. Then we study $\Afftrans^1$-twisted monogenerators before finally introducing $G$-twisted monogenerators for arbitrary groups $G$. 

There is a moduli space for each notion of twisted monogenicity. We use these moduli spaces now and defer the proof until Theorem \ref{thm:twistedisquotient}:

\begin{theorem}[{=Theorem \ref{thm:twistedisquotient}}]\label{thm:GmandAffrepablepreview}
Let $\GG_m, \Afftrans^1$ act on $\Aff^1$ on the left in the natural way, inducing a left action on $\Gen_{S'/S}$. The stack quotients $\stquot{\Gen/\GG_m}$ and $\stquot{\Gen/\Afftrans^1}$ represent $\GG_m$- and $\Afftrans^1$-twisted monogenerators up to equivalence, respectively. 
\end{theorem}

\subsection{$\GG_m$-Twisted monogenerators and Proj of the Weil Restriction}\label{ssec:Gmtwisted}

Writing $S' = \Spec B$ and  $S = \Spec A$, a twisted monogenerator amounts to a Zariski cover $S = \Spec A = \bigcup U_i$, a system of closed embeddings $\theta_i : S'_{U_i} \subseteq \Aff^1_{U_i}$ over $U_i$, and elements $a_{ij} \in \GG_m(U_i)$ such that
\[a_{ij}. \theta_j = \theta_i : S'_{U_{ij}} \to \Aff^1_{U_{ij}}.\]
Equivalently, a twisted monogenerator is a line bundle $L$ on $S$ defined by the above cocycle $a_{ij}$ and a global embedding $\theta : S' \subseteq L$ over $S$. Twisted monogenerators $(\theta_i)_{i \in I}, (\theta_j')_{j \in J}$ with respect to covers covers $\{U_i\}_{i \in I}, {U_j'}_{j \in J}$ are identified if they differ by global units $u \in \GG_m(S)$ on a common refinement of the covers $\{U_i\}_{i \in I}, {U_j'}_{j \in J}$, i.e., if the corresponding line bundles $L, L'$ are isomorphic in a way that identifies the closed embeddings $\theta, \theta'$.

For number fields $L/K$ with $\theta \in \Ints_L$ and $a \in \Ints_K$, one has $a \theta \in \Ints_L$. If $a \in \Ints_K^\ast$, then $\theta$ is a monogenerator if and only if $a\theta$ is. The multiplication action $\GG_m(\Ints_K) \action \WR(\Ints_K)$ corresponds to the global $\GG_m$ action on the vector bundle $\WR$ over $S$.

An action of $\GG_m$ corresponds to a $\ZZ$-grading on the sheaf of algebras \cite[0EKJ]{sta}. Locally in $S$, $\pi_\ast \OO_{S'} \simeq \bigoplus_{i = 1}^n \OO_S \cdot e_i$ and $\WR \simeq \Aff^n_S$. The $\GG_m$ action is the diagonal action and corresponds to the total degree of polynomials in $\OO_{\Aff^n_S} = \OO_S[x_1, \dots, x_n]$.

The associated projective bundle to the vector bundle $\WR$ is given by the relative Proj \cite[01NS]{sta}
\[\PP \WR_{S'/S} \coloneqq   \PProj{S}{\OO_{\WR}},\]
with the total-degree grading. The ideal $\indexsheaf_{S'/S}$ cutting out the complement $\NGen_{S'/S}$ of $\Gen_{1, S'/S} \subseteq \WR$ is graded by \cite[Remark 3.10]{abhs_paper_1}, defining a closed subscheme $\PP \NGen_{S'/S} \subseteq \PP \WR$. 

\begin{definition}\label{defn:projmon}
Define the \textit{scheme of projective monogenerators}
\[\PGen_{S'/S} \subseteq \PP \WR_{S'/S} \coloneqq   \PProj{S}{ \OO_{\WR_{S'/S}} }\] 
to be the open complement of the closed subscheme $\PP \NGen_{S'/S}$ cut out by the graded homogeneous ideal $\indexsheaf_{S'/S}$. 
\end{definition}

The reader may define projective polygenerators in the same fashion. 

\begin{lemma}
The vanishing of the irrelevant ideal $V(\OO_{\WR, +})$ of $\WR_{S'/S}$ is contained inside of the non-monogenicity locus $\NGen_{S'/S}$ for $S' \neq S$. 
\end{lemma}

\begin{proof}
Locally, the lemma states that $\theta = 0$ is not a monogenerator. 
\end{proof}

\begin{remark}

We relate the Proj construction to stack quotients by $\GG_m$ according to \cite[Example 10.2.8]{olssonstacksbook}. The ring $\OO_{\WR_{S'/S}}$ is generated by elements of degree one. Locally, $\WR_{S'/S} \simeq \Aff^n_S$ and $\OO_{\WR_{S'/S}} \simeq \OO_S[x_1, \dots, x_n]$ is generated by the degree one elements $x_i$. Write $\SSpec{S}{\OO_\WR}$ for the relative spectrum \cite[01LQ]{sta}. The map 
\[
\begin{tikzcd}
  \SSpec{S}{\OO_\WR} \setminus V(\OO_{\WR, +}) \ar[r, "-/\GG_m"] & \PProj{S}{\OO_\WR}
\end{tikzcd}
\]
is therefore a stack quotient or $\GG_m$-torsor. 

We have a pullback square
\[\begin{tikzcd}
\Gen_{S'/S} \ar[r] \ar[d] \pb         &\SSpec{S}{\OO_\WR} \ar[d] \setminus V(\OO_{\WR, +})         \\
\PGen_{S'/S} \ar[r]        &\PProj{S}{\OO_\WR}
\end{tikzcd}\]
of $\GG_m$-torsors and a stack quotient $\PGen_{S'/S} = \stquot{\Gen_{S'/S}/\GG_m}$. 
\end{remark}

Theorem \ref{thm:GmandAffrepablepreview} states that $\stquot{\Gen/\GG_m}$ represents twisted monogenerators, and now we know the quotient stack is wondrously a \emph{scheme}:
\begin{corollary}
The scheme $\PGen_{S'/S} = \stquot{\Gen/\GG_m}$ represents the $\GG_m$-twisted monogenerators of Definition \ref{defn:twistedmon}. That is, $\PGen_{S'/S}$ is a moduli space for twisted monogenerators. The action $\GG_m \action \Gen_{S'/S}$ is free. 
\end{corollary}

\begin{warning}
Given a monogenerator $\theta \in \Ints_L$ and a pair $a \in \Ints_K^\ast$, $\beta \in \Ints_L$, write
\[\beta = b_0 + b_1 \theta + \cdots + b_{n-1} \theta^{n-1}.\]
One may try to define a second action
\[a.\beta \overset{?}{\coloneqq  } b_0 + b_1 a \theta + b_2 a^2 \theta^2 + \cdots + b_{n-1} a^{n-1} \theta^{n-1}\]
encoding the degree with respect to $\theta$, but this action does not define a grading as it is almost never multiplicative. For example, take $\Ints_L = \ZZ[\sqrt{2}]$ with monogenerator $\sqrt{2}$ over $\ZZ_K=\ZZ$. Then
\[a.2 = 2 \neq a.\sqrt{2} \cdot a.\sqrt{2}.\]

In the case that $S' = \Spec A[\epsilon]/\epsilon^{m+1}$ and $S = \Spec A$, we recall that $\WR_{S'/S} = \jetsp{\Aff^1_A,m}$ is the jet space of $\AA^1_A$. Here, the action $a.x$ \emph{is} multiplicative and induces a second action $\GG_m \action \WR_{S'/S} = \jetsp{\Aff^1, m}$. The two actions of $\lambda \in \GG_m$ on a jet
\[f(\e) = a_0 + a_1 \e + \cdots + a_m \e^m\]
on $\Aff^1$ are $\lambda.f(\e) = \lambda \cdot f(\e)$ and $\lambda.f(\e) = f(\lambda \e)$. The Proj of $\jetsp{X, m}$ with respect to this second $\GG_m$ action is known as a ``Demailly-Semple jet'' or a ``Green-Griffiths jet'' in the literature \cite[Definition 6.1]{vojtajets}. For certain $S' \to S$, there may be a distinguished one-parameter subgroup, i.e., the image of $\GG_m \to \Aut_S(S')$, that results in a second action $\GG_m \action \WR_{S'/S}$ and allows an analogous construction. 
\end{warning}

\subsection{$\Afftrans^1$-Twisted monogenerators and affine equivalence}\label{ssec:afftwisted}

We enlarge our study to representing spaces of $\Afftrans^1$-twisted monogenerators and the related study of affine equivalence classes of ordinary monogenerators. We delay twisting by general sheaves of groups other than $\GG_m$ and $\Afftrans^1$ until the next section. For an $S$-scheme $X$, the \emph{automorphism sheaf} $\AAut(X)$ is the subsheaf of automorphisms in $\HHom_S(X, X)$. 

\begin{remark}
The automorphism sheaf $\AAut_S(\Aff^1)$ has a subgroup of affine transformations $\Afftrans^1$ under composition. These are
identified in turn with $\Aff^1 \rtimes \GG_m$ via
\[
  (a,b) \mapsto (x \mapsto bx + a).
\]

The automorphism sheaf can be much larger for other $\Aff^k_S$. For example, 
\[(x, y) \mapsto (x + y^3, y)\]
is an automorphism of $\Aff^2$. 

The automorphism sheaf $\AAut(\Aff^1)$ is \emph{not} the same as $\Afftrans^1$, though they have the same points over reduced rings. See \cite{dupuyautsA1} for some discussion over nonreduced rings. 

\end{remark}

Recall that two monogenerators $\theta_1, \theta_2$ of an $A$-algebra $B$ are said to be \emph{equivalent} if
\[
  \theta_1 = u \theta_2 + v,
\]
where $u \in A^*$ and $v \in A$. Likewise, say that two embeddings $\theta_1, \theta_2 : S' \to \mathcal{L}$ of $S'$ into an $\Afftrans^1$ bundle $\mathcal{L}$ over $S$ are \emph{equivalent}
if there is in $f \in \Afftrans^1(S)$ such that $\theta_1 = f . \theta_2$. The set of monogenerators up to equivalence is then
\[\Gamma(S, \Gen)/\Gamma(S, \Afftrans^1). \]

If $S' \overset{\sim}{\to} S$ is an isomorphism and $n=1$, the action of $\Afftrans^1$ is trivial. Otherwise, the $\Afftrans^1$-action is often free:

\begin{lemma}\label{lem:affauttrivialstabs}
The action $\AAut(S') \action \Gen_X$ has trivial stabilizers, for any quasiprojective $X$. If $S$ is normal and $S' \to S$ is not an isomorphism, the action $\Afftrans^1 \action \Gen_{S'/S}$ has trivial stabilizers as well. 

\end{lemma}

\begin{proof}

A stabilizer of the action $\Aut_S(S') \action \Gen_X$ entails a diagram 
\[\begin{tikzcd}
X \ar[r, equals]       &X      \\
S' \ar[u, hook] \ar[r, no head, "\sim"]      &S'. \ar[u, hook]
\end{tikzcd}\]    
The fact that $S' \subseteq X$ is a monomorphism forces $S' \simeq S'$ to be the identity. 

Normality of $S$ means $S$ is a finite disjoint union of integral schemes \cite[033N]{sta}; we assume $S$ is integral without loss of generality.

Computing stabilizers of $\Afftrans^1 \action \Gen_{S'/S}$ is local, so we may assume $S' \to S$ is induced by a non-identity finite map $A \to B$ of rings where $A$ is an integrally closed domain with field of fractions $K$. A stabilizer 
\[a + b \theta = \theta; \quad \quad \quad a \in A, b \in A^\ast\]
implies $(1-b) \theta = a$. If $b = 1$, then $a = 0$ and the stabilizing affine transformation is trivial. Otherwise, $1-b \in K^*$ and $\theta = \dfrac{a}{1-b} \in K$. 
Elements $\theta \in B$ are all integral over $A$. Since $A$ is integrally closed, $\theta\in A$. Hence $B = A[\theta] = A$, a contradiction.
\end{proof}

\begin{remark}
Suppose given transition functions $(a_{ij}, b_{ij})$ and local monogenerators $(\theta_i)$ as in an $\Afftrans^1$-twisted monogenerator that \emph{may not} satisfy the cocycle condition \textit{a priori}. For normal $S$ with $n > 1$ as in the lemma, the cocycle condition holds automatically, since $\Afftrans^1$ acts without stabilizers.
\end{remark}

\begin{corollary} \label{cor:stquot_is_sheaf}
If $S$ is normal, the stack quotient $\stquot{\Gen_{1, S'/S}/\Afftrans^1}$ is represented by the ordinary sheaf quotient $\Gen/\Afftrans^1$.
\end{corollary}
\begin{proof}
If $G \action X$ is a free action, the stack quotient $\stquot{X/G}$ coincides with the sheaf quotient $X/G$. 
\end{proof}

\begin{remark}\label{rmk:affineequivvstwistedmon}
If $S$ is normal, Corollary \ref{cor:stquot_is_sheaf} tells us that an $\Afftrans^1$-twisted monogenerator is the same as a global section $\Gamma(S, \Gen/\Afftrans^1)$. Equivalence classes of monogenerators are given by the presheaf quotient $\Gamma(S, \Gen)/\Gamma(S, \Afftrans^1)$. 

Affine equivalence classes of monogenerators thereby relate to twisted monogenerators in an exact sequence of pointed sets: 
\[\Gamma(S, \Gen)/\Gamma(S, \Afftrans^1) \to \Gamma(S, \Gen/\Afftrans^1) \to H^1(S, \Afftrans^1).\]
As in sheaf cohomology, the second map takes $\ol{\theta}$ to its torsor of lifts $\delta(\ol{\theta})$ in $\Gen_{S'/S}$: 
\[
  \delta(\ol{\theta})(U) = \{ f \in \Gen(U) : f + \Afftrans^1(U) = \ol{\theta}|_U \}.
\]
A section of the sheaf quotient 
\[\overline{\theta} \in \Gamma(S, \Gen/\Afftrans^1)\]
lifts to an affine equivalence class in the presheaf quotient $\theta \in \Gamma(S, \Gen)/\Gamma(S, \Afftrans^1)$ if and only if the induced $\Afftrans^1$-torsor is trivial. 

The exact sequence is analogous to Cartier divisors. If $X$ is an integral scheme with rational function field $K(X)$, the long exact sequence associated to
\[1 \to \OO_X^\ast \to K(X)^\ast \to K(X)^\ast/\OO_X^\ast \to 1\]
is analogous to the above.

\end{remark}

\begin{remark}

One can do the same with $\GG_m$, or any other group that acts freely. Compare twisted monogenerators $\PGen = \stquot{\Gen/\GG_m}$ with ordinary monogenerators $\Gen_{S'/S}$ up to $\GG_m$-equivalence to obtain a sequence
\[\Gamma(S, \Gen)/\Gamma(S, \GG_m) \to \Gamma(S, \Gen/\GG_m) \to H^1(S, \GG_m).\]
Freeness of the action is necessary to identify the  stack quotient with the ordinary sheaf quotient. 

\end{remark}

Sometimes, being $\GG_m$-twisted monogenic is the same as being $\Afftrans^1$-twisted monogenic:

\begin{proposition}\label{prop:affequivlinebund}
If $S = \Spec A$ is affine, all $\Afftrans^1$-torsors on $S$ are induced by $\GG_m$-torsors:
\[H^1(\Spec A, \GG_m) \simeq H^1(\Spec A, \Afftrans^1).\]
The corresponding twisted forms of $\Aff^1$ are the same, so we can furthermore identify $\GG_m$-twisted monogenerators with $\Afftrans^1$-twisted monogenerators. 
\end{proposition}

\begin{proof}

The maps
\[
  \begin{minipage}{2.5cm}
  \begin{center}
  $\Aff^1 \to \Afftrans^1$\\
  $a \mapsto (x \mapsto x + a)$
  \end{center}
  \end{minipage}
  \quad \text{ and } \quad
  \begin{minipage}{2.8cm}
  \begin{center}
  $\Afftrans^1 \to \GG_m$ \\
  $(x \mapsto bx + a) \mapsto b$
  \end{center}
  \end{minipage}
\]
fit into a short exact sequence
\[
  0 \to \Aff^1 \to \Afftrans^1 \to \GG_m \to 1.
\]
The sheaf $\Afftrans^1$ is \emph{not} commutative. Cohomology \emph{sets} $H^i(S, \Afftrans^1)$ are nevertheless defined for $i = 0, 1, 2$. By Serre Vanishing \cite[01XB]{sta} we have $H^i(\Spec A, \Aff^1) = 0$ for $i \neq 0$, and therefore $\Gamma(\Afftrans^1) \to \Gamma(\GG_m)$ is surjective, yielding an identification in all nonzero degrees:
\[H^i(\Spec A, \text{Aff}^1) \simeq H^i(\Spec A, \GG_m), \quad \quad i = 1, 2. \]

The action $\GG_m \action \Aff^1$ is the restriction of that of $\Afftrans^1$, factoring
\[\GG_m \subseteq \Afftrans^1 \to \AAut(\Aff^1).\]
The corresponding twisted forms of $\Aff^1$ are the same. 
\end{proof}

\subsection{Consequences of $\Afftrans^1$-twisted monogenicity}
We conclude with several consequences of twisted monogenicity and our Theorem \ref{thm:twistedClassNumb1} that shows twisted monogenerators detect class number-one number rings.

The following theorem constrains the line bundles that may be used for twisted monogenicity. This result constrains the possible Steinitz classes of a twisted monogenic extension. This is an effective constraint in geometric situations: see Lemma \ref{lem:twisted_monogenic_curve_degrees}. The structure of the set of ideals corresponding to Steinitz classes of number rings is the subject of a variety of open questions. This has traditionally been the domain of class field theory; two notable papers are \cite{McCulloh} and \cite{Cobbe}. For $n>0$, write $d(n)=\gcd\big(\{ \frac{\ell-1}{2}: \ell \text{ prime, } \ell\mid n\}\big)$. Theorems 1 and 2 of \cite{McCulloh} imply that if $K$ contains a primitive $n$th root of unity, then the Steinitz classes of Galois extensions $L/K$ of degree $n$ are precisely the $d(n)^{\text{th}}$ powers in the class group of $K$.
Compare this to the following:

\begin{theorem} \label{thm:twisted_mono_triangular_power}
Suppose $S' \to S$ is $\GG_m$-twisted monogenic, with an embedding into a line bundle $E$. Let $\mathcal{E}$ be the sheaf of sections of $E$. Then
\[
  \det(\pi_\ast \OO_{S'})\coloneqq   \wedge^{top} \pi_\ast \OO_{S'} \simeq \mathcal{E}^{-\frac{n(n-1)}{2}}
\]
in $\Pic(S)$.

In particular, if an extension of number rings $\Ints_L/\Ints_K$ is twisted monogenic, then its Steinitz class is an $\frac{n(n-1)}{2}$th power in the class group.
\end{theorem}

\begin{proof}
Write $\Sym \ast \mathcal E \coloneqq   \bigoplus_d \Sym d \mathcal E$ for the symmetric algebra. Recall that $E \simeq \VV(\cal E^\vee) \coloneqq   \SSpec{S}{\Sym\ast (\mathcal{E}^\vee) }$.
We have a surjection of $\OO_S$-modules
\[
  \Sym \ast \mathcal{E}^\vee \twoheadrightarrow \pi_\ast \OO_{S'},
\]
which we claim factors through the projection of
$\OO_S$-modules $\Sym \ast \mathcal{E}^\vee \to \bigoplus^{n-1}_{i=0} \Sym {i} \mathcal E^\vee$. Such a factorization is a local question and local factorizations automatically glue because there is at most one. Locally, we may assume $S' \to S$ is induced by a ring homorphism $A \to B$ and $\mathcal{E}^\vee$ is trivialized. We have a factorization of $A$-modules
\[\begin{tikzcd}
\Sym \ast \mathcal{E}^\vee \cong A[t] \ar[rr, two heads] \ar[dr]        && B      \\
        &\bigoplus^{n-1}_{i=0} \Sym {i} \mathcal E^\vee \simeq \bigoplus^{n-1}_{i=0} A \cdot t^i \ar[ur, dashed]
\end{tikzcd}\]
due to the existence of a monic polynomial $m_\theta(t)$ of degree $n$ for the image $\theta$ of $t$ in $\OO_{S'}$ \cite[Lemma 2.11]{abhs_paper_1}. The $A$-modules $\bigoplus^{n-1}_{i=0} A\cdot t^i$ and $B$ are abstractly isomorphic, and any surjective endomorphism of a finitely generated module is an isomorphism \cite[Theorem 2.4]{matsumura}.

We conclude that globally
\[
  \pi_\ast \OO_{S'} \simeq \bigoplus_{i=0}^{n-1} \Sym i \mathcal E^\vee.
\]
Since $\mathcal{E}$ is invertible, $\Sym i \mathcal E^\vee = (\mathcal E^\vee)^{i}$. Taking the determinant,
\[
  \det(\pi_\ast \OO_{S'}) = \det\left(\bigoplus_{i = 0}^{n - 1} (\mathcal{E}^{\vee})^{i} \right) = \mathcal{E}^{-\sum_{i = 0}^{n - 1} i} = \mathcal{E}^{-\frac{n(n-1)}{2}}.
\]
\end{proof}

The literature abounds with finiteness results on equivalence classes of monogenerators, for example:

\begin{theorem}[{\cite[Theorem 5.4.4]{EvertseGyoryBook}}] \label{thm:finiteness_monogenicity} 
Let $A$ be an integrally closed integral domain of characteristic zero and finitely generated over $\mathbb{Z}$. Let $K$ be the quotient field of $A$, $\Omega$ a finite \'etale $K$-algebra with $\Omega \neq K$, 
and $B$ the integral closure of $A$ in $\Omega$. Then there are finitely many equivalence classes of monogenic generators of $B$ over $A$.
\end{theorem}

We have an analogous finiteness result for equivalence classes of $\Afftrans^1$-twisted monogenerators:

\begin{corollary}\label{thm:fintwistedmons}
Let $A$, $K$, $\Omega$, $B$ be as in Theorem \ref{thm:finiteness_monogenicity}, with $S' \to S$ induced from $A \to B$. Assume $\Pic(A)$ is finitely generated. Then there are finitely many equivalence classes of $\Afftrans^1$-twisted monogenerators for $S' \to S$.
\end{corollary}

\begin{proof}

We essentially use the sequence
\[\Gamma(S, \Gen)/\Gamma(S, \Afftrans^1) \to \Gamma(S, \Gen/\Afftrans^1) \to H^1(S, \Afftrans^1)\]
of Remark \ref{rmk:affineequivvstwistedmon}. If this \emph{were} a short exact sequence of groups, the outer terms being finite would force the middle term to be; our proof is similar in spirit. 

Since $S$ is quasicompact and there are finitely many elements of the Picard group which are $\frac{n(n-1)}{2}$th-roots of the Steinitz class
\[\Pic(S) \coloneqq   H^1(S, \GG_m) = H^1(S, \Afftrans^1),\] 
we can find an affine open cover $S = \bigcup_i U_i$ by finitely many open sets of $S$ that simultaneously trivializes all $\frac{n(n-1)}{2}$th-roots of the Steinitz class on $S$. 

The above sequence of presheaves restricts to the $U_i$'s in a commutative diagram
\[\begin{tikzcd}
\Gamma(S, \Gen)/\Gamma(S, \Afftrans^1) \ar[r] \ar[d]      &\Gamma(S, \Gen/\Afftrans^1) \ar[r] \ar[d, hook, "\rho"]    &H^1(S, \Afftrans^1) \ar[d]        \\
\prod \Gamma(U_i, \Gen)/\Gamma(U_i, \Afftrans^1) \ar[r]    &\prod \Gamma(U_i, \Gen/\Afftrans^1) \ar[r]      &\prod H^1(U_i, \Afftrans^1).
\end{tikzcd}\]
The restriction $H^1(S, \Afftrans^1) \to \prod H^1(U_i, \Afftrans^1)$ is zero on the $\frac{n(n-1)}{2}$th-roots of the Steinitz class by construction of the $U_i$'s. The restriction $\rho : \Gamma(S, \Gen/\Afftrans^1) \hookrightarrow \prod \Gamma(U_i, \Gen/\Afftrans^1)$ is injective by the sheaf condition. A diagram chase reveals that the restriction $\rho(\overline{\theta})$ of any section $\overline{\theta} \in \Gamma(S, \Gen/\Afftrans^1)$ is in the image of $\prod \Gamma(U_i, \Gen)/\Gamma(U_i, \Afftrans^1)$. Theorem \ref{thm:finiteness_monogenicity} asserts that each set $\Gamma(U_i, \Gen)/\Gamma(U_i, \Afftrans^1)$ is finite. 
\end{proof}

\begin{lemma} \label{lem:deg2_is_twisted_mono}
Degree-two extensions are all $\Afftrans^1$-twisted monogenic. If $S$ is affine, they are also $\GG_m$-twisted monogenic. 
\end{lemma}

\begin{proof}
Localize and choose a basis containing 1 to write $\pi_\ast \OO_{S'} \simeq \OO_S \oplus \OO_S \theta_1$ for some $\theta_1 \in \Gamma(\OO_{S'})$. Given an element $\theta_2$ so that $\{1, \theta_2\}$ is also a basis, we may write 
\[\theta_1 = a + b \theta_2, \quad \quad \quad \theta_2 = c + d \theta_1, \quad \quad \quad a, b, c, d \in \OO_S. \]
Hence $bd=1$ are units, and the transition functions come from $\Afftrans^1 = \Aff^1 \rtimes \GG_m$. By choosing such generators on a cover of $S$, one obtains a twisted monogenerator. Proposition \ref{prop:affequivlinebund} further refines our affine bundle to a line bundle. 
\end{proof}

\begin{theorem}\label{thm:twistedClassNumb1}
A number ring $\Ints_K$ has class number one if and only if all twisted monogenic extensions of $\Ints_K$ are in fact monogenic. 
\end{theorem}

\begin{proof}

If the class number of $K$ is one, then all line bundles on $\Spec \Ints_K$ are trivial and the equivalence is clear. Mann \cite{Mann} has shown that $K$ has quadratic extensions without an integral basis if and only if the class number of $K$ is not one: adjoin the square root of $\alpha$, where $(\alpha)=\gb^2\gc$ with $\gb$ non-principal and $\gc$ square-free. By Lemma \ref{lem:deg2_is_twisted_mono}, such an extension is necessarily $\GG_m$-twisted monogenic. As the monogenicity of quadratic extensions is equivalent to the existence of an integral basis, the result follows. 
\end{proof}

\begin{remark}
Theorem~\ref{thm:twistedClassNumb1} implies that the ring of integers of a number field is twisted monogenic over $\mathbb{Z}$ if and only if it is monogenic over $\mathbb{Z}$. Example~\ref{ex:2genicOverZ_paper2} thus provides an example of a number field which is not twisted monogenic. 
\end{remark}

Given that twisted monogenic extensions and monogenic extensions coincide over $\ZZ$, we should ask for an example where we have twisted monogenicity but not monogenicity. All degree 2 extensions of number rings are twisted monogenic as Lemma \ref{lem:deg2_is_twisted_mono} shows. Thus every quadratic extension without an integral basis is twisted monogenic but not monogenic, and \cite{Mann} provides a construction of such extensions. The aim of the following is the very explicit construction of a higher degree example of such an extension. Though we are ultimately unable to prove non-monogenicity in the following example, we hope it gives the reader a concrete sense of the concepts and methods employed in this section.

\begin{example}[Properly Twisted Monogenic, Not Quadratic] \label{ex:TwistedMonNotQuad}
Let $K=\QQ(\sqrt[3]{5\cdot 23})$ and let $\gp_3$, $\gp_5$, and $\gp_{23}$ be the unique primes of $K$ above 3, 5, and 23, respectively. 
One can compute $\mathfrak{p}_3 = (\rho_3\coloneqq 1970(\sqrt[3]{5\cdot23})^2 + 9580(\sqrt[3]{5\cdot23}) + 46587)$. 
Consider $\ZZ_L/\ZZ_K$, where $L=K(\sqrt[3]{23  \rho_3})$. On $D(\gp_{23})$, the local index form with respect to the local basis $\{1, \sqrt[3]{23   \rho_3}, (\sqrt[3]{23  \rho_3})^2\}$ is $b^3-23\rho_3c^3$. On $D(\gp_5)$, we have the local index form $B^3-5^2\rho_3C^3$ with respect to the local basis $\{1, \sqrt[3]{5^2  \rho_3}, (\sqrt[3]{5^2  \rho_3})^2\}$. We transition via $\sqrt[3]{23^2\cdot 5^2}/23$, which is not a global unit, so the extension $\Ints_L/\Ints_K$ is twisted monogenic. 

To see what is going on more explicitly, we investigate how the transitions affect the local index forms. We have
\[b^3-23\rho_3c^3=\frac{5^2}{23}B^3-\frac{5^4}{23^2}\cdot23\rho_3C^3=\frac{5^2}{23}B^3-\frac{5^4}{23}\rho_3C^3= \text{ a unit in }  \Ocal_{D(\gp_{23})}.\]
If $B$ and $C$ could be chosen to be $\ZZ_K$-integral so that local index form represented a unit of $\ZZ_K$, then $\sqrt[3]{5^2 \rho_3}$ would be a global monogenerator. However, $\gp_{5}$-adic valuations tell us $\sqrt[3]{5^2 \rho_3}$ is not a monogenerator. One can also apply Dedekind's index criterion to $x^3-5^2\rho_3$. Similarly, we have
\[B^3-5^2\rho_3C^3=\frac{23}{5^2}b^3-\frac{23^2}{5^4}\cdot5^2\rho_3c^3=\frac{23}{5^2}b^3-\frac{23^2}{5^2}\rho_3c^3= \text{ a unit in }  \Ocal_{D(\gp_{5})}.\]
If $b$ and $c$ could be chosen to be $\ZZ_K$-integral so that local index form represented a unit of $\ZZ_K$, then $\sqrt[3]{23 \rho_3}$ would be a global monogenerator. As above, the $\gp_{23}$-adic valuations tell us this cannot be the case. Again, we could also use polynomial-specific methods.

We have shown that $\ZZ_L/\ZZ_K$ is twisted monogenic, but it remains to show that the twisting is non-trivial. We need to show the ideal $\gp_5=(5,\sqrt[3]{5\cdot 23})$ is not principal. On $D(\gp_{23})$ it can be generated by $\sqrt[3]{5\cdot 23}$ and on $D(\gp_5)$ it can be generated by $5$. We transition between these two generators via $\sqrt[3]{5^2\cdot 23^2}/23$, exactly as above. Thus our twisted monogenerators correspond to a non-trivial ideal class.

A computer algebra system can compute a $K$-integral basis for $\Ints_L$:
\begin{align*}
&\bigg\{1, \left(-2\sqrt[3]{23\cdot 5}+\frac{5}{25}\left(\sqrt[3]{23\cdot 5}\right)^2\right)\sqrt[3]{23\rho_3}+\left(-3+\frac{3}{23}\sqrt[3]{23\cdot 5}\right)\left(\sqrt[3]{23\rho_3}\right)^2,\\
&\left(120589+5243\sqrt[3]{23\cdot 5} + \frac{5243}{23}\left(\sqrt[3]{23\cdot 5}\right)^2\right)\sqrt[3]{23\rho_3} \\ 
&\quad \quad \quad\quad \quad+\left(22850+\frac{57125}{23}\sqrt[3]{23\cdot 5}+1828\left(\sqrt[3]{23\cdot 5}\right)^2\right)\left(\sqrt[3]{23\rho_3}\right)^2 \bigg\},
\end{align*}
with index form: 
\begin{align*}
\localindex_{\ZZ_L/\ZZ_K}=&13796817 (\sqrt[3]{5\cdot 23})^2 b^3 - 1367479703949 (\sqrt[3]{5\cdot 23})^2 b^2 c \\
&+ 45179341009193328 (\sqrt[3]{5\cdot 23})^2 b c^2 + 67103709 \sqrt[3]{5\cdot 23} b^3 \\
& - 497537273719431009077 (\sqrt[3]{5\cdot 23})^2 c^3- 6650125342740 \sqrt[3]{5\cdot 23} b^2 c\\
&+ 219702478196413227 \sqrt[3]{5\cdot 23} b c^2 - 2419492830176044167763 \sqrt[3]{5\cdot 23} c^3 \\
&+326269891 b^3 - 32339923090800 b^2 c + 1068411032584717260 b c^2 \\
&-11765841517121285321908 c^3.
\end{align*}

Because $\ZZ_L/\ZZ_K$ is twisted monogenic, there are no common index divisors. Thus we will always find solutions to $\localindex_{\ZZ_L/\ZZ_K}$ when we reduce modulo a prime of $\ZZ_K$. We do not expect $\ZZ_L$ to be monogenic over $\ZZ_K$; however, showing that there are no values of $b,c\in\ZZ_K$ such that $\localindex_{\ZZ_L/\ZZ_K}(b,c)\in\ZZ_K^\ast$ appears to be rather difficult. 

A clever way to get around this issue would be to show that the different of $L/K$ was non-principal. This would preclude monogenicity by prohibiting an integral basis all together. Unfortunately, one can compute that the different is principal, so the extension does have a relative integral basis. 
\end{example}

\begin{remark}
One can perform the same construction of Example~\ref{ex:TwistedMonNotQuad} with radical cubic number rings other than $\QQ(\sqrt[3]{5 \cdot 23})$. Specifically, take any radical cubic where $(3)=\gp_3^3=(\alpha)^3$, $\ell$, and $q$ are distinct primes with $(\ell)=\gl^3$, $(q)=\gq^3$, and neither $\gl$ nor $\gq$ principal. The ideas behind this construction can be taken further by making appropriate modifications.
\end{remark}


\subsection{Twisting in general}\label{ssec:generaltwisted}

Throughout this section, fix notation as in Situation \ref{sit:gensetup_paper2} and work in the category $\Sch{S}$ of schemes over $S$ equipped with the \'etale topology. In particular, we allow $X$ to be any quasiprojective $S$-scheme.

Definition \ref{defn:twistedmon} readily generalizes. Replace $\GG_m$ by any \'etale sheaf of groups $G$ with a left action $G \action X$. A $G$-twisted monogenerator for $S' \to S$ (into $X$) is an \'etale cover $U_i \to S$, closed embeddings $\theta_i : S'_{U_i} \subseteq X_{U_i}$, and elements $g_{ij} \in G(U_{ij})$ such that
\[g_{ij}.\theta_j = \theta_i : S'_{U_{ij}} \to X_{U_{ij}}.\]
Say two $G$-twisted monogenerators $(\theta_i), (\eta_i)$ are \emph{equivalent} if, after passing to a common refinement of the associated covers, there is a global section $g \in G(S)$ so that $\theta_i = g|_{U_i} \cdot \eta_i$ for all $i$.
Equivalently, the $\theta_i$'s glue to a global closed embedding $S' \subseteq \hat X$ into a twisted form $\hat X$ of $X$ the same way the $\GG_m$-twisted monogenerators give embeddings into a line bundle. 

The twisted form $\hat X$ arises from transition functions in $G$, meaning there is a $G$-torsor $P$ such that $\hat X$ is the contracted product:
\[\hat X = X \wedge^G P \coloneqq   X \times P/(G, \Delta)\]
We have already seen the variant $G = \Afftrans^1$, $X = \Aff^1$. Other interesting cases include $G = \PGL_2 \action \PP^1$, $\GL_n \action \Aff^n$, an ellliptic curve $E$ acting on itself $E \action E$, etc. 

\begin{remark}

Usually, contracted products are defined for a left action $G \action P$ and a \emph{right} action $G \action X$ by quotienting by the antidiagonal action 
\[X \wedge^G P \coloneqq   X \times P/(-\Delta, G)\]
defined by 
\[G \action X \times P; \quad \quad \quad g.(x, p) \coloneqq   (x.g^{-1}, g.p).\]
We instead take two left actions and quotient by the diagonal action of $G$. The literature often turns left actions $\GG_m \action \Aff^1$ into right actions anyway, as in \cite[Remark 1.7]{breengerbes}. 

\end{remark}

Group sheaves $G$ beget stacks $BG = BG_S$ classifying $G$-torsors on $S$-schemes with universal $G$-torsor $S \to BG$.

Twisted forms $\hat X$ of $X$ are equivalent to torsors for $\AAut(X)$ (\cite[Theorem 4.5.2]{QPointsPoonen}), as follows: Given a twisted form $\hat X \to S$, we obtain the torsor $\IIsom(\hat X, X)$ of local isomorphisms. Given a $\AAut(X)$-torsor $P$, we define a twisted form via contracted product: 
\[\hat X_P \coloneqq   X \wedge^{\AAut(X)} P\]
The stack $B\AAut(X)$ is thereby a moduli space for twisted forms of $X$ with universal family $X \wedge_{B\AAut(X)}^{\AAut(X)} S = \stquot{X/\AAut(X)}$. An action $G \to \AAut(X)$ lets one turn a $G$-torsor $P$ into a twisted form 
\[\hat X_P \coloneqq   X \wedge^G P\] 
classified by the map $BG \to B\AAut(X)$.

The automorphism sheaf $\AAut(X)$ acts on the scheme $\Gen_X$ via postcomposition with the embeddings $S' \to X$, yielding a map of sheaves
\[
  \gamma: \AAut(X) \to \AAut(\Gen_X)
  \footnote{The map $\gamma$ need not be injective. Consider $S = \Spec k$ a geometric point, $S' = \bigsqcup^3 S$ the trivial $3$-sheeted cover, and $X = \bigsqcup^2 S$ only $2$-sheeted. There are no closed immersions $S' \subseteq X$, so $\Gen_X = \varnothing$ has global automorphisms $\Aut(\varnothing) \simeq \{ id \}$, but $\Aut(X) \simeq \ZZ/2\ZZ$ is nontrivial. Similar examples abound for non-monogenic $S' \to S$. }.
\]
Similarly, the automorphism sheaf $\AAut(S')$ acts on $\Gen_X$ on the right via precomposition:
\[
  \xi: \AAut(S')^{op} \to \AAut(\Gen_X).
\]

The induced map $H^1(S, \AAut(X)) \to H^1(S, \AAut(\Gen_X))$ sends a twisted form $\widehat{X}$ of $X$ to the twisted form $\Gen_{\widehat{X}}$ of $\Gen_X$ given by looking at closed embeddings into $\widehat{X}$. We package these twisted forms $\Gen_{\hat X}$ into a universal version: $\TMon$. 

\begin{definition}[$\TMon$]
Let $\TMon \to B\Aut(X)$ be the $B\Aut(X)$-stack whose $T$-points are given by 
\[\left\{\hspace{-.1 cm}\begin{tikzcd}
        &\TMon \ar[d]      \\
T \ar[r, "\widehat{X}", swap] \ar[ur, dashed]       &B \Aut(X)
\end{tikzcd}\hspace{-.1 cm}\right\} \hspace{-.05 cm}\coloneqq \hspace{-.05 cm}
\left\{\hspace{-.17 cm}\begin{tikzcd}
S' \times_S T \ar[dr] \ar[rr, "s", hook, dashed]       &       &\widehat{X} \ar[dl]       \\
        &T
\end{tikzcd} \middle|\, \, s \text{ is a closed immersion}\right\}.\]    
\end{definition}

The map $\TMon \to B \Aut(X)$ is representable by schemes, since pullbacks are twisted forms $\Gen_{\hat X}$ of $\Gen_X$ itself:
\[\begin{tikzcd}
\Gen_{\hat X} \ar[r] \ar[d] \pb        &S \ar[d, "\hat X"]       \\
\TMon \ar[r]        &B \AAut(X).
\end{tikzcd}\]

The universal torsor over $B\Aut(X)$ is $S$, but the universal twisted form is obtained by the contracted product with $X$ over $B \Aut(X)$: 
\[X \wedge^{\Aut(X)}_{B\Aut(X)} S \simeq \stquot{X/\Aut(X)}.\]
One can exhibit $\TMon$ as an open substack of the Weil Restriction of $\stquot{X/\Aut(X)} \to B\Aut(X)$ as in \cite[Proposition 2.3]{abhs_paper_1}. There is a universal closed embedding over $\TMon$ into the universal twisted form of $X$ as in the definition of $\Gen_{S'/S}$:
\[\begin{tikzcd}
\TMon \times_S S' \ar[rr, dashed, "u", hook] \ar[dr]       &       &\stquot{X/\Aut(X)}. \ar[dl]     \\
        &\TMon
\end{tikzcd}\]

The universal case is concise to describe but unwieldy because $\AAut(X)$ need not be finite, smooth, or well-behaved in any sense. We simplify by specifying our twisted form $\widehat{X} \to T$ to get a scheme $\TMon_{\widehat{X}} = T \times_{\widehat{X}, B\Aut(X)} \TMon$ or by specifying the structure group $G$. 

Fixing the structure group $G$ requires $\widehat{X} = X \wedge^G P$ for the specified sheaf of groups $G$ and some $G$-torsor $P$. These $G$-twisted forms are parameterized by the pullback
\[\begin{tikzcd}
\TMon^G  \ar[r] \ar[d] \pb        &\TMon \ar[d]      \\
BG \ar[r]      &B\Aut(X).
\end{tikzcd}\]
The fibers of $\TMon^G \to BG$ over maps $T \to BG$ are again twisted forms of $\Gen_X$. If $X = \Aff^k$ and $\widehat{X}$ is a $G$-twisted form, we may refer to the existence of sections of $\TMon_{\widehat{X}}$ by saying $S'/S$ is $G$-twisted $k$-genic, etc.

\begin{remark}
Trivializing $\hat X$ is not the same as trivializing the torsor $P$ that induces $\hat X$ unless the group $G$ is $\Aut(X)$ itself. For example, take the trivial action $G \action X$.
\end{remark}

\begin{theorem}\label{thm:twistedisquotient}
The stack of twisted monogenerators $\TMon$ is isomorphic to $\stquot{\Gen_X/\AAut(X)}$ over $B\AAut(X)$. More generally, for any sheaf of groups $G \action X$, we have an isomorphism $\TMon^G \simeq \stquot{\Gen_X/G}$ over $BG$. 
\end{theorem}

\begin{proof}

Address the second, more general assertion and let $T$ be an $S$-scheme. Write $X_T \coloneqq X \times_S T$, $T' \coloneqq  T \times_S S'$, etc. A $T$-point of $\TMon^G$ is a $G$-torsor $P \to T$ and a solid diagram 
\[\begin{tikzcd}
P' \ar[d, dashed] \ar[rr, dashed] \pb      &&X_T \times_T P \ar[d, "{/\Delta, G}"]        \\
T' \ar[rr] \ar[dr]      &\phantom{a}&X_T \wedge^G_T P,  \ar[dl]     \\
        &T
\end{tikzcd}\]
with $T' \to X_T \wedge_T^G P$ a closed immersion. Form $P'$ by pullback: $P'$ is a left $G$-torsor with an equivariant map to $X_T \times_T P$ with the diagonal action.
The map $P'\to X_T\times_T P$ entails a pair of equivariant maps $P' \to P$ and $P' \to X_T$. The map $P' \to P$ over $T$ forces $P' \simeq P \times_T T'$. These data form an equivariant map $P \to \Gen_X$ over $T$, or $T \to \stquot{\Gen_X/G}$. Reverse the process to finish the proof. 
\end{proof}

To see this theorem in practice, we have the following example.

\begin{definition}\label{def:affanddaff}
The group sheaf $\Afftrans^k \subseteq \AAut(\Aff^k)$ of \emph{affine transformations} is the set of functions
\[
  \vec{x} \mapsto M \vec{x} + \vec{b}
\]
where $M \in \GL_k$ and $\vec{b} \in \Aff^k$, under composition. Note $\Afftrans^k \cong \Aff^k \rtimes \GL_k$. 
\end{definition}

\begin{example}
Let $S'$ be the spectrum of a quadratic order such as $\ZZ[i]$, and take $S = \Spec \ZZ$. The space of $k$-generators is $\Gen_{k,S'/S} = \Aff^k \times (\Aff^k \setminus 0)$ according to \cite[Proposition 4.5]{abhs_paper_1}. Take the quotient by the groups of affine transformations:
\[[\Gen_{k,S'/S}/\Aff^k \rtimes \GG_m] = \PP^{k-1}, \quad \quad \quad [\Gen_{k,S'/S}/\Afftrans^k] = [\Aff^k/\GL_k] = [\PP^{k-1}/\PGL_k].\]

These quotients represent twisted monogenerators according to Theorem \ref{thm:twistedisquotient}. The corresponding $\PGL_k$-torsors were classically
identified with Azumaya algebras and Severi-Brauer varieties (see the exposition in \cite{severibrauerkollar}) or twisted forms of $\PP^{k-1}$. These yield classes in the Brauer group $H^2(\GG_m)$ via the connecting homomorphism from
\[1 \to \GG_m \to \GL_k \to \PGL_k \to 1.\]

The same holds locally for any degree-two extension $S' \to S$ with $S$ integral using \cite[Proposition 4.5]{abhs_paper_1}.  
\end{example}

If $G$ is an abelian variety over a number field $S = \Spec K$, let $P \to S$ be a $G$-torsor inducing a twisted form $\hat X$ of $X$. Given a twisted monogenerator $\theta : S' \subseteq \hat{X}$, one can try to promote $\theta$ to a global monogenerator by trivializing $P$ and thus $X$. 

Suppose one is given trivializations of $P$ over the completions $K_\nu$ at each place. Whether these glue to a global trivialization of $P$ over $K$ and thus a monogenerator $S' \subseteq X$ is governed by the Shafarevich-Tate group $\Sha(G/K)$. 

Given a $G$-twisted monogenerator with local trivializations, the Shafarevich-Tate group obstructs lifts of $\theta$ to a global monogenerator the same way classes of line bundles in $\Pic$ obstruct $\GG_m$-twisted monogenerators from being global monogenerators. Theorem \ref{thm:twistedClassNumb1} showed a converse---nontrivial elements of $\Pic$ imply twisted monogenerators that are not global monogenerators. 
\begin{question}
Is the same true for $\Sha$? Does every element of the Shafarevich-Tate group arise this way? 

\end{question}

The Shafarevich-Tate group approach is useless for $G = \GG_m$ or $\GL_n$ because of Hilbert's Theorem 90 \cite[03P8]{sta}: 
\[H^1_{et}(\Spec K, G) = H^1_{Zar}(\Spec K, G) = 0.\]
The same goes for any ``special'' group with \'etale and Zariski cohomology identified. The strategy may work better for $\PGL_n$ or elliptic curves $E$.

\begin{remark}

This section defined $G$-twisted monogenerators using covers in the \'etale topology, whereas Definition \ref{defn:twistedmon} used the Zariski topology. For $G = \GG_m$ or $\Afftrans^1$, either topology gives the same notion of twisted monogenerators. Observe that $\GG_m$ has the same Zariski and \'etale cohomology by Hilbert's Theorem 90. The same is true for $\Aff^1$ by \cite[03P2]{sta} and so also $\Afftrans^1 = \GG_m \ltimes \Aff^1$. 

\end{remark}


\section{Examples of the scheme of monogenerators}\label{sec:examples}

We conclude with several examples to further illustrate the interaction of the various forms of monogenicity considered in this paper. 
We will make frequent reference to computation of the index form using the techniques of the previous paper in this series. Some of these examples were already considered in the previous paper, but are revisited in order to add some commentary on their relationship to notions of local monogenicity.

\subsection{Orders in number rings}

\begin{example}[Dedekind's Non-Monogenic Cubic Field]\label{ex:Dedekind_paper2}
Let $\eta$ denote a root of the polynomial $X^3 - X^2 - 2X -8$ and consider the field extension $L\coloneqq  \QQ(\eta)$ over $K\coloneqq \QQ$. When Dedekind constructed this example \cite{Dedekind} it was the first example of a non-monogenic extension of number rings. Indeed two generators are necessary to generate $\Ints_L/\Ints_K$: take $\eta^2$ and $\frac{\eta + \eta^2}{2}$, for example. In fact, $\{1,\frac{\eta + \eta^2}{2},\eta^2\}$ is a $\ZZ$-basis for $\Ints_K$. 
The matrix of coefficients with respect to the basis $\{1,\frac{\eta + \eta^2}{2},\eta^2\}$ is

\[\begin{bmatrix}
1 & a & a^{2} + 6b^{2} + 16bc + 8 c^{2} \\
0 & b & 2 a b + 7 b^{2} + 24bc + 20c^{2} \\
0 & c & -2b^{2} + 2ac - 8bc - 7c^{2}
\end{bmatrix}.\]

Taking its determinant, the index form associated to this basis is \[
 -2b^3 - 15b^2c - 31bc^2 - 20c^3
.\]
Were the extension monogenic, we would be able to find $a,b,c \in \ZZ$ so
that the index form above is equal to $\pm 1$.

In fact, $\ZZ_{\QQ(\eta)}/\ZZ$ is not even locally monogenic. By Lemma \ref{prop:mono_over_points_index_form}, we may check by reducing at primes. Over the prime $2$ the index form reduces to
\[
  b^2c + bc^2,
\]
and iterating through the four possible values of $(b,c) \in (\ZZ/2\ZZ)^2$ shows that the index form always to reduces to 0. That is, $2$
is a common index divisor.

Dedekind showed that $\ZZ_{\QQ(\eta)}/\ZZ$ is non-monogenic, not by using an index form, but by deriving a contradiction from the $\ZZ_L$-factorization of the ideal 2, which splits into three primes. In our terms, $\Spec \ZZ_{\QQ(\eta)} \to \Spec \ZZ$ has three points over $\Spec \mathbb{F}_2$, all with residue field $\mathbb{F}_2$. Therefore, condition \eqref{artinmono2} of Theorem \ref{thm:artin_monogenicity} for monogenicity at the prime $(2)$ fails, so $S' \to S$ is not monogenic.
\end{example}

With base extension, one can eventually resolve the obstructions presented by common index divisors. The following example illustrates a non-maximal order where we have a slightly different obstruction to monogenicity.

\begin{example}[An order that fails to be monogenic over geometric points]\label{ex:NonMonOrderInMonogenicExt_paper2}

Consider the extension $\ZZ[\sqrt{2}, \sqrt{3}]$ of $\ZZ$. (This is not the maximal order of $\QQ(\sqrt{2},\sqrt{3})$.) We recall from \cite[Example 4.13]{abhs_paper_1} that the index form with respect to the basis $\{ 1, \sqrt{2}, \sqrt{3}, \sqrt{6} \}$ is
\[	
=-4(2b^2 - 3c^2)(b^2 - 3d^2)(c^2 - 2d^2).
\]

By Proposition \ref{thm:monogenic_over_geometric_points},
$S' \to S$ fails to be monogenic, even over geometric points, since the index form reduces to 0 in the fiber over 2. In terms of Theorem \ref{thm:local_artin_monogenicity}, $S' \to S$ fails to be monogenic since the fiber over $2$ consists of a single point with a two dimensional tangent space.
\end{example}

We can contrast the above examples with the following example where the obstruction to monogenicity is global as opposed to local.

\begin{example}[A Zariski-locally monogenic, but not twisted-monogenic extension]\label{ex:2genicOverZ_paper2}
Here we take a closer look at one member of the family in Example \ref{ex: NoCommonIndexNotMono}. Let $K = \QQ$, $L = K(\sqrt[3]{5^2\cdot7})$. The ring of integers $\Ints_L = \ZZ[\sqrt[3]{5^2\cdot 7}, \sqrt[3]{5\cdot 7^2}]$ is not monogenic over $\ZZ$. Let $\alpha= \sqrt[3]{5^2\cdot 7}$, $\beta = \sqrt[3]{5\cdot 7^2}$. We recall from \cite[Example 4.14]{abhs_paper_1} that $\{1, \alpha, \beta\}$ is a $\ZZ$-basis for $\ZZ_L$ and the associated index form is $5b^3 - 7c^3$. Thus, for a given choice of $a,b,c \in \ZZ$, the primes which divide the value $5b^3 - 7c^3$ are precisely the primes at which $a + b\alpha + c\beta$ will fail to
generate the extension.

The values obtained by this index form are $\{0,\pm5\}$ modulo $7$. Since 5 is a unit in $\ZZ/7\ZZ$, we do have local monogenerators over $D(7)$. Similarly, we have local monogenerators over $D(5)$. Together, $D(5)$ and $D(7)$ form an open cover, and we see that this extension is Zariski-locally monogenic. However, as we can see by reducing modulo 7, the index form cannot be equal to $\pm 1$. Therefore there are no global monogenerators and $\ZZ_L/\ZZ$ is not monogenic: in the language of \cite{AlpogeBhargavaShnidman}, $\ZZ_L/\ZZ$ has a local obstruction to monogenicity, despite being locally monogenic.

This is not a $\mathbb{G}_m$-twisted monogenic extension by Theorem \ref{thm:twistedClassNumb1}, because $h(\ZZ) = 1$ and this extension is not globally monogenic. See Example \ref{ex:TwistedMonNotQuad} for a properly $\GG_m$-twisted monogenic extension and a comparison to the example presented here.
\end{example}

\subsection{Maps of curves}

One benefit of our more geometric notion of monogenicity is it allows us readily ask questions about monogenicity in classical geometric situations with the same language that we use in the arithmetic context.
Our next examples concern the case that $S' \to S$ is a finite map of algebraic curves, which is
essentially never monogenic. On the other hand, we find explicit examples of $\GG_m$-twisted monogenic $S' \to S$. Theorem \ref{thm:twisted_mono_triangular_power} constrains the possible line bundles that we may use to show $\GG_m$-twisted monogenicity. We make this precise in the lemma below.

\begin{lemma} \label{lem:twisted_monogenic_curve_degrees}
Let $\pi : C \to D$ be a finite map of smooth projective curves of degree $n$ and let $g$ denote the genus.
\begin{enumerate}
  \item\label{monoCD} $\pi$ is only monogenic if it is the identity map;
  \item\label{twistedmonoCD} If $\pi$ is $\GG_m$-twisted monogenic, then $1 - g(C) - n(1 - g(D))$ is divisible by $\frac{1}{2}n(n-1)$ in $\ZZ$. Moreover, if $\pi$ factors through a closed embedding into a line bundle $E$ with sheaf of sections $\mathcal{E}$, then
  \[
    \deg(\mathcal{E}) = - \frac{1 - g(C) - n(1 - g(D))}{\frac{1}{2}n(n-1)}.
  \]
\end{enumerate}

\end{lemma}
\begin{proof}
To see \eqref{monoCD}, note that a map $f : C \to \AA^1_D$ is determined by a global section of $\OO_C$. Since $C$ is a proper variety, the global sections of $\OO_C$ are constant functions. It follows that a map $f : C \to \AA^1_D$ is constant on fibers of $\pi$. Therefore $f$ cannot be an immersion unless $\pi$ has degree 1, i.e., is the identity.

Suppose $\pi : C \to D$ is $\GG_m$-twisted monogenic with an embedding
into a line bundle $E$. By \cite[0AYQ]{sta} and Riemann-Roch,
\[
  \deg(\det(\pi_*\OO_C)) = 1 - g(C) - n(1 - g(D)).
\]
By Theorem \ref{thm:twisted_mono_triangular_power},
\[
  \det(\pi_*\OO_{S'}) \simeq \mathcal{E}^{-\frac{n(n-1)}{2}}
\]
where $\mathcal{E}$ is the sheaf of sections of $E$. Taking degrees
of both sides,
\[
  1 - g(C) - n(1 - g(D)) = -\deg(\mathcal{E})\cdot\frac{n(n-1)}{2}.
\]
This shows \eqref{twistedmonoCD}.
\end{proof}

First, we will investigate one of the most basic families of maps of curves.

\begin{example}[Maps $\PP^1 \to \PP^1$]
Let $k$ be an algebraically closed field and let $\pi : \mathbb{P}^1_k \to \mathbb{P}^1_k$ be a finite map of degree $n$. If $n = 1$, then $\pi$ is trivially monogenic.
When $n = 2$, Lemma \ref{lem:twisted_monogenic_curve_degrees} tells us that $\pi$ cannot be monogenic, while Lemma \ref{lem:deg2_is_twisted_mono} tells us that $\pi$ is $\Afftrans^1$-twisted monogenic.
Lemma \ref{lem:twisted_monogenic_curve_degrees} tells us that for degrees $n > 2$ the map $\pi$ is neither monogenic nor $\GG_m$-twisted monogenic, although Theorems \ref{thm:locally_mono_is_mono_over_points} and \ref{thm:artin_monogenicity} tell us that $\pi$ is Zariski-locally monogenic. 

Working with $\ZZ$ instead of an algebraically closed field, consider the map $\pi : S' = \PP^1_{\ZZ} \to S = \PP^1_{\ZZ}$ given by $[a : b] \mapsto [a^2 : b^2]$. We will show by direct computation that this map is $\GG_m$-twisted monogenic. Write $U = \Spec \ZZ[x]$ and $V = \Spec \ZZ[y]$ for the standard affine charts of the target $\PP^1$. The map $\pi$ is then given
on charts by
\begin{align*}
  \ZZ[x] \to &\ZZ[a] \\
    x  \mapsto & \, a^2
\end{align*}
and
\begin{align*}
  \ZZ[y] \to &\ZZ[b] \\
    y  \mapsto & \, b^2.
\end{align*}

Let us compute $\Gen_{1,S'/S}$. Over $U$, $\pi_* \OO_{\PP^1}$ has $\ZZ[x]$-basis $\{ 1, a \}$. Let $c_1, c_2$ be the coordinates of $\WR|_U = \Aff^2$, with universal map 
\[\ZZ[c_1, c_2, t] \to \ZZ[c_1, c_2, a], \quad \quad \quad t \mapsto c_1 + c_2 a.\]

The index form associated to this basis is
\[
  \localindex(c_1, c_2) = c_2.
\]
Similarly, $\pi_*\OO_{\PP^1}$ has $\ZZ[y]$-basis $\{1, b\}$, $\WR|_V$ analogous coordinates $d_1, d_2$, and the index form associated to this basis is
\[
  \localindex(d_1, d_2) = d_2.
\]
An element of $\ZZ[x]$ (resp. $\ZZ[y]$) is a unit if and only if it is $\pm 1$,
so
\begin{align*}
  \Gen_{S'/S}(U) &= \{ c_1 \pm a \mid c_1 \in \ZZ[x] \} \\
  \Gen_{S'/S}(V) &= \{ d_1 \pm b \mid d_1 \in \ZZ[y] \}.
\end{align*}
We can see directly that $\pi$ is not monogenic: the condition that a monogenerator $c_1 \pm a$ on $U$ glue with a monogenerator $d_1 \pm b$ on $V$ is that
\[
  (c_1 \pm a)|_{U \cap V} = (d_1 \pm b)|_{U \cap V}.
\]
But this is impossible to satisfy since $a|_{U \cap V} = b|_{U \cap V}^{-1}$.

Lemma \ref{lem:twisted_monogenic_curve_degrees} tells us that if $S' \to S$ is twisted monogenic, the line bundle into which $S'$ embeds must have degree 1. Let us therefore attempt to embed $S'$ into the line bundle with sheaf of sections $\OO_{\PP^1}(1)$. The sheaf $\OO_{\PP^1}(1)$ restricts to the trivial line bundle on both $U$ and $V$, and a section $f \in \OO_U$ glues to a section $g \in \OO_V$ if
\[
  y \cdot f|_{U \cap V} = g|_{U \cap V}.
\]
Embedding $S'$ into this line bundle is therefore equivalent to finding a monogenerator $c_1 \pm a$ on $U$, and a monogenerator $d_1 \pm b$ on $V$ such that
\[
  y \big((c_1 \pm a)|_{U \cap V}\big) = (d_1 \pm b)|_{U \cap V}.
\]
Bearing in mind that $y = b^2 = a^{-2}$ on $U \cap V$, we find a solution by taking positive signs, $c_1 = 0$, and $d_1 = 0$. Therefore $\pi: \PP^1_\ZZ \to \PP^1_\ZZ$ is twisted monogenic.

\end{example}

Lemma \ref{lem:twisted_monogenic_curve_degrees} tells us that we must pass to higher genus to find a $\GG_m$-twisted monogenic cover of $\PP^1$ of degree greater than $2$. Here is an example where the source is an elliptic curve.

\begin{example}[Twisted monogenic cover of degree 3]
Let $E$ be the Fermat elliptic curve $V(x^3 + y^3 - z^3) \subset \mathbb{P}^2_\ZZ$. Consider the projection from $[0 : 0 : 1]$, i.e.,
the map $\pi : E \to \PP^1$ defined by $[x : y : z] \mapsto [x : y]$. Write $U = \Spec \ZZ[x]$ and $V = \Spec \ZZ[y]$ for the standard affine charts
of $\PP^1$.
The map is given on charts by $\ZZ[x] \to \ZZ[x, z]/(x^3 + 1 - z^3)$
and $\ZZ[y] \to \ZZ[y,yz]/(1 + y^3 - (yz)^3)$. The gluing on overlaps is given
by $x \mapsto y^{-1}$ on $\PP^1$ and by $x \mapsto y^{-1}, z \mapsto z$ on $E$.

We now compute $\Gen_{E/\PP^1}$. Note that over $U$, $\OO_E$ has the $\ZZ[x]$-basis $1, z, z^2$. 
The index form
associated to $1, z, z^2$ is
\[
  \localindex(c_1, c_2, c_3) = c_2^3 - c_3^3(x^3 + 1).
\]

Over $V$, $\OO_E$ has the $\ZZ[x]$-basis $1, yz, y^2z^2$. The index form
associated to this basis is
\[
  \localindex(d_1, d_2, d_3) = d_2^3 - d_3^3(y^3 + 1).
\]

An element of $\ZZ[x]$ or $\ZZ[y]$ is a unit if and only if it is $\pm 1$. This implies that
\begin{align*}
  \Gen_{S'/S}(U) &= \{ c_1 \pm z \mid c_1 \in \ZZ[x] \} \\
  \Gen_{S'/S}(V) &= \{ d_1 \pm yz \mid d_1 \in \ZZ[y] \}.
\end{align*}
We see that there are no global sections of $\Gen_{S'/S}$, since coefficients of $z$
cannot match on overlaps.

However, if we twist so that we are considering embeddings of $E$ into $\OO_{\PP^1}(1)$,
then the condition for a monogenerator $c_1 \pm z$ on $U$
to glue with a monogenerator $d_1 \pm yz$ on $V$ is that
\[
  y\big((c_1 \pm z)|_{U \cap V}\big) = (d_1 \pm yz)|_{U \cap V}. 
\]
This is satisfied, for example by taking the positive sign for both generators and $c_1 = d_1 = 0$. Therefore $E \to \PP^1$ is twisted monogenic
with class $1 \in \Pic(\PP^1_{\ZZ})$.
\end{example}

\bibliographystyle{amsalpha}
\bibliography{zbib}

\end{document}